\documentclass[10pt,reqno]{amsart}
\usepackage{amsmath,amssymb,amsthm,upref,mathrsfs,enumerate}

\numberwithin{equation}{section}
\allowdisplaybreaks

\newtheorem{theorem}{Theorem}[section]
\newtheorem{proposition}[theorem]{Proposition}
\newtheorem{lemma}[theorem]{Lemma}
\newtheorem{corollary}[theorem]{Corollary}

\theoremstyle{definition}
\newtheorem{definition}[theorem]{Definition}
\newtheorem{remark}[theorem]{Remark}

\newcommand{\T}{\mathcal T}
\newcommand{\Z}{{\mathbb Z}}
\newcommand{\R}{\mathbb R}
\newcommand{\N}{\mathbb N}
\newcommand{\D}{\mathscr D}
\newcommand{\Sscr}{\mathscr S}
\newcommand{\Fscr}{\mathscr F}
\newcommand{\M}{\mathcal M}
\newcommand{\A}{\mathcal A}
\newcommand{\one}{\mathbf 1}
\newcommand{\avg}[1]{\left\langle #1\right\rangle}
\newcommand{\Mloc}{\mathcal M^{\mathrm{loc}}}
\newcommand{\norm}[2][]{\left\lVert #2\right\rVert_{#1}}

\usepackage{color}

\title[Vector-valued maximal inequalities on ball Banach function spaces]
{Vector-Valued Hardy--Littlewood Maximal Inequalities on Ball Banach Function Spaces}


\begin{document}

\author{Yoshihiro Sawano}
\address{Department of Mathematics,
Graduate School of Science and Engineering,
Chuo University,
1-13-27 Kasuga,
Bunkyo-ku,
Tokyo 112-8551, Japan}
\email{yoshihiro-sawano@celery.ocn.ne.jp}

\author{Mitsuo Izuki}
\address{Faculty of Liberal Arts and Sciences,
Tokyo City University,
1-28-1 Tamazutsumi,
Setagaya-ku,
Tokyo 158-8557, Japan}
\email{izuki@tcu.ac.jp}

\author{Takahiro Noi}
\address{Department of Mathematical and Data Science,
Otemon Gakuin University,
2-1-15 Nishiai,
Ibaraki,
Osaka 567-8502, Japan}
\email{taka.noi.hiro@gmail.com}

\begin{abstract}
The goal of this paper is to characterize
the boundedness property of the uncentered Hardy--Littlewood maximal operator
on ball Banach function spaces and their K\"{o}the associate spaces 
in terms of the vector-valued maximal inequality.
The main result of this paper supplements the recent results
by Nieraeth.
\end{abstract}
\maketitle

\section{Introduction}\label{sec 20260701-1}

The goal of this paper is to investigate the boundedness property
of operators acting on Banach lattices.
Let $\M$ denote the uncentered Hardy--Littlewood maximal operator over
axis-parallel cubes in $\R^n$.
See \eqref{eq 20260702-4} for the precise definition. For $1<q<\infty$, we consider whether
there is a constant $C_{q,X}$, independent of the positive integer $N$,
such that
\begin{equation}\label{eq 20260701-1}
 \left\|\left(\sum_{j=1}^N(\M f_j)^q\right)^{\frac1q}\right\|_X
 \leq C_{q,X}
 \left\|\left(\sum_{j=1}^N|f_j|^q\right)^{\frac1q}\right\|_X
\end{equation}
for every $N\in\N$ and every finite family $\{f_j\}_{j=1}^N$ of
measurable functions for which the quantity on the right is finite.
Here we use the following notation:
Throughout the paper, for a measurable function $f$ we allow
$\int_Q |f(x)|\,dx$ to take the value $+\infty$, and we put
\begin{equation}\label{eq 20260702-4}
 \M f(x)=\sup\limits_{Q}\one_Q(x)\avg{|f|}_Q,
 \qquad
 \avg{f}_Q=\frac1{|Q|}\int_Q f,
\end{equation}
where the supremum is over all axis-parallel cubes. 
Here it will be understood that $0\cdot\infty =0$.
Thus $\M f$ is
initially an extended nonnegative measurable function. If $f$ is locally
integrable, all the averages in \eqref{eq 20260702-4} are finite.

 The
constant in \eqref{eq 20260701-1} is required to be independent of $N$.
This is the finite-sequence form of the classical vector-valued maximal
inequality of Fefferman and Stein \cite{FeffermanStein1971}. We also use its
weighted form due to Andersen and John \cite{AndersenJohn1981}.
Throughout the paper, a constant in a vector-valued estimate is only
required to be finite and independent of $N$. 

Let $X$ be a ball Banach function space 
(see Definition \ref{def 20260702-1})
and let $X'$ be its K\"othe
associate space. Consider the two boundedness assumptions
\begin{equation}\label{eq 20260701-2}
 \M:X\to X,
 \qquad
 \M:X'\to X'.
\end{equation}
It is known that these assumptions imply the vector-valued estimate
\eqref{eq 20260701-1} for every $q\in(1,\infty)$. This implication
follows from \cite[Theorem~1.1]{INS2026} by applying the extrapolation
theorem there to the weighted vector-valued maximal inequality of
Andersen and John \cite{AndersenJohn1981}. As explained in
\cite{INS2026}, that extrapolation theorem is a special case of
\cite[Theorem~4.7 and Remark~4.8]{Nieraeth2023}.

The relation between \eqref{eq 20260701-2} and sparse averaging operators
is also known. Lorist and Nieraeth \cite[Lemma~3.4]{LoristNieraethCompact2024}
and Nieraeth \cite[Theorem~1.2]{Nieraeth2026} proved that
\eqref{eq 20260701-2} is equivalent to uniform boundedness of sparse
averaging operators. We use this characterization in the proof of
Theorem~\ref{thm 20260701-1}, after deriving the required uniform sparse
bound from \eqref{eq 20260701-1}. Lerner obtained another criterion for
the boundedness of $\M$ on $X'$ in terms of local quantiles; see
\cite{Lerner2025}.

A related result concerns the linearized cube-average operators
\[
 T_Qf=\avg{f}_Q\one_Q.
\]
Here and below $\one_E$ denotes the indicator function of a set $E$.
Suppose that $X$ is $r_0$-convex for some $r_0>1$. Let $\mathcal Q$
denote the countable collection of cubes with rational corners.
Nieraeth considers the sequence-valued map
\[
 \{f_Q\}_{Q\in\mathcal Q}
 \longmapsto
 \{\avg{f_Q}_Q\one_Q\}_{Q\in\mathcal Q}.
\]
Here the norm of a sequence $\{g_Q\}_{Q\in\mathcal Q}$ in
$X[\ell^r(\mathcal Q)]$ is
\[
 \left\|\{g_Q\}_{Q\in\mathcal Q}\right\|_{X[\ell^r(\mathcal Q)]}
 =\left\|\left(\sum_{Q\in\mathcal Q}|g_Q|^r\right)^{1/r}\right\|_X.
\]
By \cite[Theorem~3.21]{Nieraeth2026}, \eqref{eq 20260701-2} is
equivalent to the boundedness of this map on $X[\ell^r(\mathcal Q)]$
for some $r\in(1,r_0]$. Restricting this bound to sequences supported
on a finite set $\mathcal F\subset\mathcal Q$ gives
\begin{equation*}\tag{$\mathrm{A}_r$}\label{eq 20260701-linearized}
 \left\|\left(\sum_{Q\in\mathcal F}|T_Qf_Q|^r\right)^{1/r}\right\|_X
 \leq C
 \left\|\left(\sum_{Q\in\mathcal F}|f_Q|^r\right)^{1/r}\right\|_X
\end{equation*}
for every finite collection $\mathcal F$ of cubes with rational corners
and every family $\{f_Q\}_{Q\in\mathcal F}\subset X$. The number $r$
and the constant $C$ are independent of $\mathcal F$ and of the family
$\{f_Q\}_{Q\in\mathcal F}$.

Let $\mathcal F=\{Q_1,\ldots,Q_N\}$ be an arbitrary finite
collection of cubes, and let $\{f_j\}_{j=1}^N\subset X$ be an arbitrary
family. For every $j$, choose cubes $Q_j^{(k)}$ with rational corners
whose lower-left corners and side lengths converge to those of $Q_j$.
The choices can be made so that the cubes $Q_1^{(k)},\ldots,Q_N^{(k)}$
are pairwise distinct for each $k$ and all the cubes $Q_j$ and
$Q_j^{(k)}$ are contained in one fixed bounded cube. Then
\[
 |Q_j^{(k)}\mathbin{\triangle}Q_j|\longrightarrow0,
 \qquad
 |Q_j^{(k)}|\longrightarrow |Q_j|,
 \qquad
 \one_{Q_j^{(k)}}\longrightarrow\one_{Q_j}
 \quad\text{almost everywhere}.
\]
Since $X\subset L^1_{\mathrm{loc}}(\R^n)$, each $f_j$ is integrable on
that fixed bounded cube. Hence the absolute continuity of the integral
gives
\[
 \left|\int_{Q_j^{(k)}}f_j-\int_{Q_j}f_j\right|
 \leq\int_{Q_j^{(k)}\mathbin{\triangle}Q_j}|f_j|
 \longrightarrow0.
\]
Together with $|Q_j^{(k)}|\to|Q_j|$, this yields
\[
 \avg{f_j}_{Q_j^{(k)}}\longrightarrow\avg{f_j}_{Q_j}
 \qquad(1\leq j\leq N).
\]
Consequently,
\[
 G_k:=\left(\sum_{j=1}^N
 |T_{Q_j^{(k)}}f_j|^r\right)^{1/r}
 \longrightarrow
 G:=\left(\sum_{j=1}^N|T_{Q_j}f_j|^r\right)^{1/r}
 \quad\text{almost everywhere}.
\]
Apply \eqref{eq 20260701-linearized} to
$\{Q_1^{(k)},\ldots,Q_N^{(k)}\}$ and the corresponding functions
$\{f_1,\ldots,f_N\}$. For $m\in\N$, set
\[
 H_m=\inf_{k\geq m}G_k.
\]
Then $0\leq H_m\uparrow G$ almost everywhere, and $H_m\leq G_k$ for
all $k\geq m$. By the lattice property of $X$,
\[
 \norm[X]{H_m}\leq\inf_{k\geq m}\norm[X]{G_k}.
\]
The Fatou property of $X$ therefore gives
\[
 \norm[X]{G}
 =\sup\limits_{m\in\N}\norm[X]{H_m}
 \leq\liminf_{k\to\infty}\norm[X]{G_k}
 \leq C\left\|\left(\sum_{j=1}^N|f_j|^r\right)^{1/r}\right\|_X.
\]
Thus \eqref{eq 20260701-linearized} also holds for arbitrary finite
collections of cubes. Nieraeth's proof of the equivalence in
\cite[Theorem~3.21]{Nieraeth2026} between \eqref{eq 20260701-2} and the
boundedness of the sequence-valued cube-average map on
$X[\ell^r(\mathcal Q)]$ uses a theorem of Rutsky \cite{Rutsky2016}.
For $r=2$, Nieraeth also proved that the implication
from \eqref{eq 20260701-linearized} to \eqref{eq 20260701-2} holds for
order-continuous Banach function spaces with the Fatou property, without a
convexity assumption; see \cite[Theorem~A and Remark~3.22]{Nieraeth2026}.
Since $|T_Qf_Q|\leq\M f_Q$ pointwise, \eqref{eq 20260701-1} with $q=2$
implies \eqref{eq 20260701-linearized} with $r=2$. Therefore Nieraeth's result
already gives the converse to \eqref{eq 20260701-1} in the
order-continuous case when $q=2$.

To the best of our knowledge, the present paper removes the convexity and
order-continuity assumptions
from this converse implication. More precisely, we prove that if
\eqref{eq 20260701-1} holds for some $q\in(1,\infty)$ with a constant
independent of $N$, then \eqref{eq 20260701-2} holds without assuming
$r_0$-convexity for any $r_0>1$ or order continuity. The
main new step is to obtain a uniform scalar sparse bound from
\eqref{eq 20260701-1}. First, the disjoint sets in a sparse family give a
bound for a sparse $\ell^q$ operator. Next, a binomial iteration changes
this $\ell^q$ bound into a bound for the ordinary sparse averaging
operator. The finite dyadic sparse criterion in
Proposition~\ref{prop 20260704-1} below then gives
\eqref{eq 20260701-2}, and in particular the boundedness of $\M$ on
$X'$.

\begin{theorem}[Vector-valued maximal characterization]\label{thm 20260701-1}
Let $X$ be a ball Banach function space on $\R^n$, and let $X'$ be
its K\"othe associate space.
\begin{enumerate}
 \item[$(i)$] The following statements are equivalent.
 \begin{enumerate}
 \item[$(a)$] $\M:X\to X$ and $\M:X'\to X'$ are bounded.
 \item[$(b)$] There exist $q\in(1,\infty)$ and a constant $C>0$ such that,
 for every $N\in\N$ and every finite family $\{f_j\}_{j=1}^N$ of
 measurable functions, the finite vector-valued inequality
 \begin{equation}\label{eq 20260701-3}
  \left\|\left(\sum_{j=1}^N(\M f_j)^q\right)^{\frac1q}\right\|_X
  \leq C
  \left\|\left(\sum_{j=1}^N|f_j|^q\right)^{\frac1q}\right\|_X
 \end{equation}
 holds
 whenever the quantity on the right is finite. The constant $C$ is
 independent of $N$ and of the family $\{f_j\}_{j=1}^N$.
 \item[$(c)$] For every $q\in(1,\infty)$, there is a constant $C>0$ such that
 \eqref{eq 20260701-3} holds for every $N\in\N$, with $C$ independent
 of $N$ and of the family $\{f_j\}_{j=1}^N$.
 \end{enumerate}
 \item[$(ii)$] The parameter $q\in(0,\infty)$ must satisfy $q>1$
if
there exists a finite constant $C$ 
 satisfying the inequality \eqref{eq 20260701-3} for every $N\in\N$
 and every finite family $\{f_j\}_{j=1}^N$ of
 measurable functions. 
 More precisely, for $0<q\leq1$ such an estimate fails on every
 ball quasi-Banach function space on $\R^n$.
 \item[$(iii)$] Fix $1<q<\infty$. Suppose that the inequality \eqref{eq 20260701-3}, with this value of $q$ and constant $C$, holds
 for every $N\in\N$. Then, for every sequence
 $\{f_j\}_{j=1}^{\infty}$ satisfying
 \[
 \left(\sum_{j=1}^{\infty}|f_j|^q\right)^{\frac1q}\in X,
 \]
 one has
 \[
 \left\|\left(\sum_{j=1}^{\infty}(\M f_j)^q\right)^{\frac1q}\right\|_X
 \leq C
 \left\|\left(\sum_{j=1}^{\infty}|f_j|^q\right)^{\frac1q}\right\|_X.
 \]
 Conversely, this estimate for all sequences implies
 \eqref{eq 20260701-3} by taking $f_j=0$ for $j>N$, with the same
 constant $C$.
 \item[$(iv)$] The following statements are equivalent. The constant $C$ may be
 chosen to be the same in $(a)$--$(c)$.
 \begin{enumerate}
 \item[$(a)$] There is a constant $C>0$ such that
 \[
  \norm[X]{\M f}\leq C\norm[X]{f}
  \qquad(f\in X).
 \]
 \item[$(b)$] There is a constant $C>0$ such that, for every $N\in\N$ and
 every finite family $\{f_j\}_{j=1}^N$ satisfying
 $\max\limits_{1\leq j\leq N}|f_j|\in X$,
 \begin{equation}\label{eq 20260701-4}
  \left\|\max\limits_{1\leq j\leq N}\M f_j\right\|_X
  \leq C\left\|\max\limits_{1\leq j\leq N}|f_j|\right\|_X.
 \end{equation}
 \item[$(c)$] There is a constant $C>0$ such that, for every sequence
 $\{f_j\}_{j=1}^{\infty}$ with $\sup\limits_{j\geq1}|f_j|\in X$,
 \[
  \left\|\sup\limits_{j \in{\mathbb N}}\M f_j\right\|_X
  \leq C\left\|\sup\limits_{j \in{\mathbb N}}|f_j|\right\|_X.
 \]
 \end{enumerate}
These equivalent conditions need not imply that $\M$ is bounded on $X'$.
\end{enumerate}
\end{theorem}
The characterization in Theorem~\ref{thm 20260701-1} is also useful
for concrete ball Banach function spaces; see, for example,
\cite[Theorem~4.12]{WanYangZhao2025}.

Let $s\in(1,\infty)$ and let $X$ be a Banach lattice. We say that
$X$ is \emph{$s$-concave} if there exists a constant $C>0$ such that,
for every finite family $\{x_j\}_{j=1}^N\subset X$,
\[
\left(\sum_{j=1}^N \|x_j\|_X^s\right)^{1/s}
\le
C
\left\|
\left(\sum_{j=1}^N |x_j|^s\right)^{1/s}
\right\|_X.
\]
Lerner recently proved that, if $X$ is an $s$-concave Banach
function space for some $s\in(1,\infty)$ and $\M$ is bounded on
$X$, then $\M$ is also bounded on $X'$; see
\cite[Corollary~1.4]{Lerner2026}. Our result is different in nature.
Without assuming $r_0$-convexity for any $r_0>1$, $s$-concavity,
or order continuity, Theorem~\ref{thm 20260701-1} $(i)$ characterizes the boundedness of $\M$ on $X$ and $X'$ by the finite vector-valued maximal inequality \eqref{eq 20260701-3}. The
principal new contribution is the converse implication: the finite
vector-valued inequality itself implies the boundedness of $\M$ on
$X'$, without any additional convexity, concavity, or
order-continuity assumption.

We briefly describe the converse argument. Let $\Sscr$ be a finite sparse
subcollection of a dyadic grid and define
\[
 \T_{\Sscr,q}f
 =\left(\sum_{Q\in\Sscr}\avg{|f|}_Q^q\one_Q\right)^{\frac1q},
 \qquad
 \A_{\Sscr}f
 =\sum_{Q\in\Sscr}\avg{|f|}_Q\one_Q.
\]
The estimate in \eqref{eq 20260701-3} first gives a bound for
$\T_{\Sscr,q}$ that is independent of $\Sscr$. The binomial iteration in
Theorem~\ref{thm 20260703-1}, combined with this bound, gives the uniform
bound for $\A_{\Sscr}$ stated in Corollary~\ref{cor 20260703-1}. The
finite dyadic sparse criterion in Proposition~\ref{prop 20260704-1} then
yields the boundedness of $\M$ on $X'$.

To obtain the result in the quasi-Banach setting, let $Y$ be a ball
quasi-Banach function space and let $r>0$. Following the notation in \cite{Nieraeth2023}, define
\begin{equation}\label{eq 20260701-5}
 Y^r=\{h:\ |h|^{1/r}\in Y\},
 \qquad
 \norm[Y^r]{h}=\norm[Y]{|h|^{1/r}}^r.
\end{equation}

The following is a standard concavification fact; see
\cite[Definition~2.7 and the discussion following it]{Nieraeth2023}. 
\begin{lemma}[Standard concavification facts]
\label{lem 20260812-6-noi}
Let $Y$ be an $r$-convex ball quasi-Banach function space with convexity
constant one, and put $X=Y^r$. Then $X$ is a Banach lattice with the Fatou
property and contains the indicator of every ball. Moreover,
\begin{equation}\label{eq 20260812-38-noi}
 \|H\|_Y^r=\|H^r\|_X
 \qquad(H\geq0).
\end{equation}
\end{lemma}

We assume that $Y$ is $r$-convex with constant one
(see Definition \ref{def 20260705-2}). Under this
assumption, the quantity in \eqref{eq 20260701-5} is a norm. Put
$X=Y^r$. For a finite collection $\Sscr$ of cubes and
$f\in L^r_{\mathrm{loc}}(\R^n)$, define
\begin{equation}\label{eq 20260701-6}
 \M_r f=\bigl[\M(|f|^r)\bigr]^{1/r},
 \qquad
 \A_{\Sscr,r}f
 =\bigl[\A_{\Sscr}(|f|^r)\bigr]^{1/r}.
\end{equation}
The following theorem is a direct rescaling
consequence of Theorem~\ref{thm 20260701-1}.

\begin{theorem}[Powered vector-valued maximal characterization]\label{thm 20260701-2}
Let $Y$ be an $r$-convex ball quasi-Banach function space with
convexity constant one for some $r>0$, and put $X=Y^r$ with the norm in
\eqref{eq 20260701-5}.
\begin{enumerate}
 \item[$(i)$] The following statements are equivalent.
 \begin{enumerate}
 \item[$(a)$] $\M:X\to X$ and $\M:X'\to X'$ are bounded.
 \item[$(b)$] There exist $q>r$ and a constant $C>0$ such that, for every
 $N\in\N$ and every finite family
 $\{f_j\}_{j=1}^N\subset L^r_{\mathrm{loc}}(\R^n)$,
 \begin{equation}\label{eq 20260701-7}
  \left\|\left(\sum_{j=1}^N(\M_r f_j)^q\right)^{\frac1q}\right\|_Y
  \leq C
  \left\|\left(\sum_{j=1}^N|f_j|^q\right)^{\frac1q}\right\|_Y
 \end{equation}
 whenever the quantity on the right is finite. The constant $C$ is
 independent of $N$ and of the family $\{f_j\}_{j=1}^N$.
 \item[$(c)$] For every $q>r$, there is a constant $C>0$ such that
 \eqref{eq 20260701-7} holds for every $N\in\N$ and every finite family
 $\{f_j\}_{j=1}^N\subset L^r_{\mathrm{loc}}(\R^n)$ for which the
 quantity on the right is finite, with $C$ independent of $N$ and of
 the family.
 \item[$(d)$] There is a constant $C_{\mathrm{sp}}>0$ such that
 \begin{equation}\label{eq 20260702-17a}
  \norm[Y]{\A_{\Sscr,r}f}
  \leq C_{\mathrm{sp}}\norm[Y]{f}
 \end{equation}
 for every finite $1/2$-sparse subcollection $\Sscr$ of every dyadic
 grid and every $f\in Y\cap L^r_{\mathrm{loc}}(\R^n)$.
 \end{enumerate}
 Each of these conditions implies
 $Y\subset L^r_{\mathrm{loc}}(\R^n)$. Moreover, for $q>r$, a constant
 $C$ works in \eqref{eq 20260701-7} if and only if $C^r$ works in
 \eqref{eq 20260701-3} on $X$ with $q/r$ in place of $q$.

 \item[$(ii)$] Let $q>0$. Suppose that there is a finite constant $C$ such that
 \eqref{eq 20260701-7} holds for every $N\in\N$ and every finite
 family $\{f_j\}_{j=1}^N\subset L^r_{\mathrm{loc}}(\R^n)$ for which
 the quantity on the right is finite, with $C$ independent of $N$ and
 of the family. Then necessarily $q>r$ and
 $Y\subset L^r_{\mathrm{loc}}(\R^n)$. In particular, no estimate in
 \eqref{eq 20260701-7} can hold for every $N\in\N$ when
 $0<q\leq r$.

 \item[$(iii)$] Fix $q>r$ and suppose that \eqref{eq 20260701-7}, with constant
 $C$, holds for every $N\in\N$. Then, for every sequence
 $\{f_j\}_{j=1}^{\infty}$ satisfying
 $\bigl(\sum_{j=1}^{\infty}|f_j|^q\bigr)^{\frac1q}\in Y$, one has
 \[
 \left\|\left(\sum_{j=1}^{\infty}(\M_r f_j)^q\right)^{\frac1q}\right\|_Y
 \leq C
 \left\|\left(\sum_{j=1}^{\infty}|f_j|^q\right)^{\frac1q}\right\|_Y.
 \]
 Conversely, this estimate for all sequences implies
 \eqref{eq 20260701-7} by taking $f_j=0$ for $j>N$, with the same
 constant $C$.

 \item[$(iv)$] Suppose that $Y\subset L^r_{\mathrm{loc}}(\R^n)$. The following
 statements are equivalent. The constant $C$ may be chosen to be the same
 in $(a)$--$(c)$.
 \begin{enumerate}
 \item[$(a)$] There is a constant $C>0$ such that
 \[
  \norm[Y]{\M_r f}\leq C\norm[Y]{f}
  \qquad(f\in Y).
 \]
 \item[$(b)$] There is a constant $C>0$ such that, for every $N\in\N$ and every
 finite family $\{f_j\}_{j=1}^N$ satisfying
 $\max\limits_{1\leq j\leq N}|f_j|\in Y$,
 \[
  \left\|\max\limits_{1\leq j\leq N}\M_r f_j\right\|_Y
  \leq C\left\|\max\limits_{1\leq j\leq N}|f_j|\right\|_Y.
 \]
 \item[$(c)$] There is a constant $C>0$ such that, for every sequence
 $\{f_j\}_{j=1}^{\infty}$ with $\sup\limits_{j \in{\mathbb N}}|f_j|\in Y$,
 \[
  \left\|\sup\limits_{j \in{\mathbb N}}\M_r f_j\right\|_Y
  \leq C\left\|\sup\limits_{j \in{\mathbb N}}|f_j|\right\|_Y.
 \]
 \end{enumerate}
\end{enumerate}
\end{theorem}

To see the relation between the two theorems, let
$f_j\in L^r_{\mathrm{loc}}(\R^n)$, set $h_j=|f_j|^r$, and put
$a=q/r$. Then
\begin{align*}
 \left\|\left(\sum_{j=1}^N(\M_r f_j)^q\right)^{\frac1q}\right\|_Y^r
 &=\left\|\left(\sum_{j=1}^N(\M h_j)^a\right)^{1/a}\right\|_X,\\
 \left\|\left(\sum_{j=1}^N|f_j|^q\right)^{\frac1q}\right\|_Y^r
 &=\left\|\left(\sum_{j=1}^N|h_j|^a\right)^{1/a}\right\|_X.
\end{align*}
Thus, after raising both sides to the power $r$,
\eqref{eq 20260701-7} becomes \eqref{eq 20260701-3} on $X$, with
$a=q/r$ in place of $q$. Similarly,
\[
 \norm[Y]{\A_{\Sscr,r}f}^r
 =\norm[X]{\A_{\Sscr}(|f|^r)}.
\]
These identities reduce the proof of Theorem~\ref{thm 20260701-2} to
Theorem~\ref{thm 20260701-1}, after the completeness and local-integrability
properties of $X$ have been verified.

The paper is organized as follows. Section~\ref{sec 20260702-1} contains
the definitions and known results used throughout the paper.
Section~\ref{sec 20260703-1} derives uniform sparse bounds from
\eqref{eq 20260701-3} under the assumption that this estimate holds for
some $q\in(1,\infty)$. Section~\ref{sec 20260704-1} recalls the finite
dyadic sparse criterion needed here and then proves
Theorem~\ref{thm 20260701-1}. Section~\ref{sec 20260705-1} is devoted to
the proof of Theorem~\ref{thm 20260701-2}.

Our results also extend to the local Hardy--Littlewood maximal operator
by using Theorem~\ref{thm 20260703-1}. Finally,
Section~\ref{sec 20260812-2-noi} establishes the local versions of
Theorems~\ref{thm 20260701-1} and~\ref{thm 20260701-2} by applying the
results of Section~\ref{sec 20260703-1}.

\section{Ball function spaces, maximal operators, and sparse families}\label{sec 20260702-1}

Section \ref{sec 20260702-1} collects preliminary facts.
\subsection{The function-space setting}

Fix $n\in\N$. We work on $\R^n$ with Lebesgue measure. For
$1<a<\infty$, write $a'=a/(a-1)$. The operators used below are
positively homogeneous but need not be linear, and they may initially
have the value $+\infty$. If $T$ assigns an extended measurable function
$Tf$ to every element $f$ of a quasi-normed function space $Z$, we set
\begin{equation}\label{eq 20260702-1}
 \norm[Z\to Z]{T}
 :=\sup\limits_{0\neq f\in Z}\frac{\norm[Z]{Tf}}{\norm[Z]{f}}
 \in[0,\infty],
\end{equation}
where the quotient is interpreted as $+\infty$ when $Tf\notin Z$.
Thus $T:Z\to Z$ is bounded exactly when the quantity in
\eqref{eq 20260702-1} is finite.

We invoke the notion of ball Banach function spaces from \cite{SHYY17}.
\begin{definition}\label{def 20260702-1}
Let $X$ be a normed linear space of measurable functions on $\R^n$,
where functions that agree almost everywhere are identified. We call
$X$ a \emph{normed function lattice} if it satisfies property $(i)$
below. We call $X$ a \emph{ball Banach function space} if it is complete
and satisfies properties $(i)$--$(iv)$.
\begin{enumerate}
 \item[$(i)$] If $f\in X$ and $|g|\leq |f|$ almost everywhere, then
 $g\in X$ and $\norm[X]{g}\leq\norm[X]{f}$.
 \item[$(ii)$] If $f_k\in X$, $0\leq f_k\uparrow f$ almost everywhere, and
 $\sup\limits_k\norm[X]{f_k}<\infty$, then $f\in X$ and
 \[
 \norm[X]{f}=\sup\limits_k\norm[X]{f_k}.
 \]
 \item[$(iii)$] For every ball $B\subset\R^n$, one has $\one_B\in X$.
 \item[$(iv)$] For every ball $B\subset\R^n$, there is a constant $C_B$
 such that
 \[
 \int_B|f(x)|\,dx\leq C_B\norm[X]{f}
 \qquad(f\in X).
 \]
\end{enumerate}
\end{definition}

We also present
the definition of ball quasi-Banach function spaces.
\begin{definition}\label{def 20260702-2}
A quasi-Banach space $Y$ of measurable functions on $\R^n$ is called
a \emph{ball quasi-Banach function space} if the following properties hold.
\begin{enumerate}
 \item[$(i)$] If $f\in Y$ and $|g|\leq |f|$ almost everywhere, then
 $g\in Y$ and $\norm[Y]{g}\leq\norm[Y]{f}$.
 \item[$(ii)$] If $f_k\in Y$, $0\leq f_k\uparrow f$ almost everywhere, and
 $\sup\limits_k\norm[Y]{f_k}<\infty$, then $f\in Y$ and
 \[
 \norm[Y]{f}=\sup\limits_k\norm[Y]{f_k}.
 \]
 \item[$(iii)$] For every ball $B\subset\R^n$, one has $\one_B\in Y$.
\end{enumerate}
\end{definition}

Let $X$ be a normed function lattice that contains the indicator of every
bounded measurable set. Its K\"othe associate space is
\begin{equation}\label{eq 20260702-2}
 X'=\left\{g:\ \norm[X']{g}:=
 \sup\limits_{\norm[X]{f}\leq1}\int_{\R^n}|f(x)g(x)|\,dx<\infty\right\}.
\end{equation}
Since $\one_E\in X$ for every bounded measurable set $E$, each $g\in X'$
is locally integrable. More precisely,
\[
 \int_E|g(x)|\,dx\leq\norm[X]{\one_E}\norm[X']{g}<\infty.
\]
The generalized H\"older inequality is
\begin{equation}\label{eq 20260702-3}
 \int_{\R^n}|f(x)g(x)|\,dx
 \leq\norm[X]{f}\norm[X']{g}.
\end{equation}

\begin{lemma}\label{lem 20260702-associate-fatou}
Let $X$ be a normed function lattice and let $X'$ be defined by
\eqref{eq 20260702-2}. If $0\leq g_k\uparrow g$ almost everywhere and
$\sup\limits_k\norm[X']{g_k}<\infty$, then $g\in X'$ and
\[
 \norm[X']{g}=\sup\limits_k\norm[X']{g_k}.
\]
In particular, $X'$ has the Fatou property.
\end{lemma}

\begin{proof}
Put $C=\sup\limits_k\norm[X']{g_k}$. For every $f\in X$, the monotone
convergence theorem and \eqref{eq 20260702-3} give
\[
 \int_{\R^n}|f(x)|g(x)\,dx
 =\lim_{k\to\infty}\int_{\R^n}|f(x)|g_k(x)\,dx
 \leq C\norm[X]{f}.
\]
Thus $g\in X'$ and $\norm[X']{g}\leq C$. The reverse inequality follows
from $g_k\leq g$ and the lattice property of $X'$.
\end{proof}

We collect some fundamental properties of ball Banach function spaces. 
\begin{lemma}\label{lem 20260702-2}
Let $X$ be a ball Banach function space.
\begin{enumerate}
 \item[$(i)$] If $E$ is a bounded measurable set, then $\one_E\in X$.
 \item[$(ii)$] Both $X$ and $X'$ are continuously embedded into
 $L^1(E)$ for every bounded measurable set $E$.
 \item[$(iii)$] If $E$ has positive measure, then there is a measurable subset
 $F\subset E$ such that $|F|>0$ and $\one_F\in X$.
\end{enumerate}
\end{lemma}

\begin{proof}
Let $E$ be bounded and choose a ball $B$ containing $E$. Since
$\one_E\leq\one_B$, part $(i)$ follows from
Definition~\ref{def 20260702-1}$(i)$,$(iii)$. For $f\in X$,
\[
 \int_E|f(x)|\,dx\leq\int_B|f(x)|\,dx\leq C_B\norm[X]{f},
\]
so the embedding for $X$ in part $(ii)$ follows from
Definition~\ref{def 20260702-1}$(iv)$.
Let us establish that $X'$ is embedded into $L^1(E)$
in part $(ii)$. If $g\in X'$, then
\[
 \int_E|g(x)|\,dx
 \leq\norm[X]{\one_E}\norm[X']{g},
\]
which proves the estimate for $X'$. Finally, if $|E|>0$, then
$|E\cap B(0,R)|>0$ for some $R>0$. Part $(i)$, applied to
$F=E\cap B(0,R)$, proves part $(iii)$.
\end{proof}

We recall \cite[Theorem 3.6]{LoristNieraeth2024} and \cite[Theorem 71.1]{Zaanen1967}.
\begin{lemma}{\rm \cite[Theorem 3.6]{LoristNieraeth2024} and \cite[Theorem 71.1]{Zaanen1967}}\label{lem 20260812-1-noi}
Let $X$ be a ball Banach function space. Then $X'$ is a ball Banach
function space, both $X$ and $X'$ are continuously embedded into $L^1(E)$
for every bounded measurable set $E$, and $X''=X$ with equality of norms.
\end{lemma}

We do not recall the proof but we make some comments.
Part $(iii)$ of Lemma~\ref{lem 20260702-2}
is precisely the saturation property used in the standard theory
of Banach function spaces. 
Together with the Fatou property in
Definition~\ref{def 20260702-1}$(ii)$, the Lorentz--Luxemburg theorem gives
$X''=X$ with equality of norms.

\subsection{Maximal and vector-valued operators}

The extrapolation theorem in \cite{INS2026} is stated for a centered maximal
operator over Euclidean balls
$\mathcal{M}_{\rm center}$. The operator $\mathcal{M}_{\rm center}$
and the operator in
\eqref{eq 20260702-4} dominate each other pointwise up to constants depending
only on $n$. Consequently, the boundedness on $X$ and $X'$, and the
estimates obtained by extrapolation, are unchanged apart from dimensional
constants.

A weight is a locally integrable function $w$ such that
$0<w<\infty$ almost everywhere. For $1<p<\infty$, the Muckenhoupt
class $A_p$ consists of the weights satisfying
\begin{equation}\label{eq 20260702-5}
 [w]_{A_p}
 =\sup\limits_Q\avg{w}_Q
 \left(\avg{w^{-1/(p-1)}}_Q\right)^{p-1}
 <\infty.
\end{equation}
We write
\[
 \norm[L^p(w)]{f}
 =\left(\int_{\R^n}|f(x)|^p w(x)\,dx\right)^{1/p}.
\]
When $w\equiv1$, we simply write $L^p=L^p(\R^n)$ and $\|f\|_{L^p}$ in place of $L^p(w)$ and $\|f\|_{L^p(w)}$, respectively.

The following weighted vector-valued estimate is due to Andersen and John
\cite{AndersenJohn1981}.

\begin{theorem}\label{thm 20260702-1}
Let $1<p,q<\infty$. There is an increasing function
$N_{n,p,q}:[1,\infty)\to(0,\infty)$ such that, for every
$w\in A_p$, every $N\in\N$, and every finite family
$\{f_j\}_{j=1}^N$ of measurable functions for which the quantity on
the right-hand side of the following inequality is finite,
\begin{equation}\label{eq 20260702-6}
 \left\|\left(\sum_{j=1}^N(\M f_j)^q\right)^{\frac1q}\right\|_{L^p(w)}
 \leq N_{n,p,q}([w]_{A_p})
 \left\|\left(\sum_{j=1}^N|f_j|^q\right)^{\frac1q}\right\|_{L^p(w)}.
\end{equation}
The function $N_{n,p,q}$ is independent of $N$.
\end{theorem}

We use the following known extrapolation theorem in the form stated in
\cite[Theorem~1.1]{INS2026}. That paper also explains that this statement
is a special case of \cite[Theorem~4.7 and Remark~4.8]{Nieraeth2023}.

\begin{theorem}\label{thm 20260702-2}
Let $X$ be a ball Banach function space such that $\M$ is bounded on
both $X$ and $X'$. Fix $p_0\in(1,\infty)$, and let
$N:[1,\infty)\to(0,\infty)$ be increasing. Let $\mathcal F$ be a
collection of pairs $(f,g)$ of nonnegative measurable functions. Assume
that, for every $w\in A_{p_0}$ and every $(f,g)\in\mathcal F$ with
$g\in L^{p_0}(w)$, one has $f\in L^{p_0}(w)$ and
\[
 \norm[L^{p_0}(w)]{f}
 \leq N([w]_{A_{p_0}})\norm[L^{p_0}(w)]{g}.
\]
Then there is a constant $C=C(X,p_0,N)$ such that, for every
$(f,g)\in\mathcal F$ with $g\in X$, one has $f\in X$ and
\[
 \norm[X]{f}\leq C\norm[X]{g}.
\]
The constant $C$ is independent of $\mathcal F$ and of the particular
pair $(f,g)$.
\end{theorem}
The weighted
spaces in \cite{Nieraeth2023} are written in the multiplier-weight convention
$L_v^{p_0}=\{H:vH\in L^{p_0}\}$. Taking $v=w^{1/p_0}$ gives
\[
 \|H\|_{L_v^{p_0}}=\|H\|_{L^{p_0}(w)},
 \qquad
 [v]_{p_0}=[w]_{A_{p_0}}^{1/p_0}.
\]
Accordingly, \cite[Theorem~4.7]{Nieraeth2023}, applied with
$\phi(t)=N(t^{p_0})$, yields the conclusion of
Theorem~\ref{thm 20260702-2} for each pair
$(f,g)\in\mathcal F$. Remark~4.8 in \cite{Nieraeth2023}
shows that the constant may be chosen independently of the particular
pair. Hence a specific constant works uniformly for all
$(f,g)\in\mathcal F$ with $g\in X$.

Let $Z$ be a ball quasi-Banach function space and let
$0<q<\infty$. Given a constant $C>0$, we say that the finite
vector-valued estimate holds on $Z$ with constant $C$ if, for every
$N\in\N$,
\begin{equation}\label{eq 20260702-7}
 \left\|\left(\sum_{j=1}^N(\M f_j)^q\right)^{\frac1q}\right\|_Z
 \leq C
 \left\|\left(\sum_{j=1}^N|f_j|^q\right)^{\frac1q}\right\|_Z
\end{equation}
whenever the quantity on the right is finite. The constant $C$ is
independent of $N$ and of the family. Here $\M f_j$ is the
extended-valued function in \eqref{eq 20260702-4}. Whenever
\eqref{eq 20260702-7} holds with a finite constant $C$, taking $N=1$
gives
\begin{equation}\label{eq 20260702-8}
 \norm[Z]{\M f}\leq C\norm[Z]{f}
 \qquad(f\in Z).
\end{equation}
Hence $\M f$ is finite almost everywhere for every $f\in Z$. If
$\displaystyle\|f\|_{L^1(Q)}=\infty$ for some cube $Q$, then
$\M f=\infty$ on $Q$, which is impossible. Therefore every element of
$Z$ is locally integrable whenever \eqref{eq 20260702-7} holds with a
finite constant independent of $N$.

\begin{lemma}\label{lem 20260702-3}
Let $X$ be a ball quasi-Banach function space,
and let $0<q<\infty$. Assume that \eqref{eq 20260702-7}, with
constant $C$, holds for every $N\in\N$. Then
$X\subset L^1_{\mathrm{loc}}(\R^n)$ and
\begin{equation}\label{eq 20260812-29-noi-global}
 \left\|
 \left(\sum_{j=1}^{\infty}({{\M}} f_j)^q\right)^{1/q}
 \right\|_X
 \leq C
 \left\|
 \left(\sum_{j=1}^{\infty}|f_j|^q\right)^{1/q}
 \right\|_X
\end{equation}
for every sequence for which the quantity on the right belongs to $X$.
\end{lemma}

\begin{proof}
If $\displaystyle\|f\|_{L^1(Q)}=\infty$ for some cube $Q$, then
\begin{equation}\label{eq:260813-1-Sawano}
\M f=\infty
\end{equation} 
on $Q$, which is impossible. Hence
$\displaystyle\|f\|_{L^1(Q)}<\infty$ for every such cube. Since every bounded set $E$ is
contained in a finite union of these cubes, say
$E\subset\bigcup\limits_{\nu=1}^m Q_\nu$, one has
\[
 \int_E|f(x)|\,dx
 \leq\sum_{\nu=1}^m\int_{Q_\nu}|f(x)|\,dx<\infty.
\]
Thus $X\subset L^1_{\mathrm{loc}}(\R^n)$.

Next, we justify the passage from 
inequality \eqref{eq 20260702-7} for finitely many functions to 
inequality \eqref{eq 20260812-29-noi-global}
for infinite sequences.
Since $|f_j|\leq F$, Definition~\ref{def 20260702-2}$(i)$ gives $f_j\in {{X}}$, and
the assumed local embedding shows that each $f_j$ is locally integrable.
For $N\geq1$, put
\[
 F_N=\left(\sum_{j=1}^{N}|f_j|^q\right)^{\frac1q},
 \qquad
 G_N=\left(\sum_{j=1}^{N}(\M f_j)^q\right)^{\frac1q}.
\]
Also set
\[
 F=\left(\sum_{j=1}^{\infty}|f_j|^q\right)^{\frac1q},
 \qquad
 G=\left(\sum_{j=1}^{\infty}(\M f_j)^q\right)^{\frac1q}.
\]
Then $F_N\leq F$, and therefore
\[
 \norm[X]{G_N}\leq C\norm[X]{F_N}
 \leq C\norm[X]{F}.
\]
Since $G_N\uparrow G$ pointwise, Definition~\ref{def 20260702-2}$(ii)$ gives
$G\in X$ and
\[
 \norm[X]{G}=\sup\limits_{N \in {\mathbb N}}\norm[X]{G_N}\leq C\norm[X]{F}.
\]
\end{proof}

We follow \cite[p. 75, 5.2]{Stein1993} to prove the following proposition.
\begin{proposition}\label{prop 20260702-1}
Let $Y$ be a ball quasi-Banach function space on $\R^n$, and let
$0<q\leq1$. There is no constant $C$ such that
\begin{equation}\label{eq 20260702-9}
 \left\|\left(\sum_{j=1}^N(\M f_j)^q\right)^{\frac1q}\right\|_Y
 \leq C
 \left\|\left(\sum_{j=1}^N|f_j|^q\right)^{\frac1q}\right\|_Y
\end{equation}
for every $N\in\N$ and all measurable functions for which the quantity
on the right is finite.
\end{proposition}

\begin{proof}
Suppose, to the contrary, that \eqref{eq 20260702-9} holds with a constant
$C$ independent of $N$. Choose $L=8M$ with $M\in\N$, and set
\[
 Q=[0,1)^n.
\]
The $L^n$ half-open cubes
\[
 Q_\nu=\prod_{i=1}^n
 \left[\frac{\nu_i-1}{L},\frac{\nu_i}{L}\right),
 \qquad
 \nu=(\nu_1,\ldots,\nu_n)\in\{1,\ldots,L\}^n,
\]
are pairwise disjoint and their union is exactly $Q$. Put
$f_\nu=\one_{Q_\nu}$. Thus we have a family of $N=L^n$ functions, and
\begin{equation}\label{eq 20260702-10}
 \left(\sum_{{{\nu \in \{1,2,\ldots,L\}^n}}}|f_\nu|^q\right)^{\frac1q}=\one_Q
 \quad\text{pointwise on }\R^n.
\end{equation}

Let $E=(3/8,5/8)^n$. For each $x\in E$, define
\[
 \kappa_i=\lfloor Lx_i\rfloor+1
 \qquad(1\leq i\leq n),
 \qquad
 \kappa=(\kappa_1,\ldots,\kappa_{n}).
\]
Since $0<x_i<1$, one has $\kappa_i\in\{1,\ldots,L\}$ and
\[
 \frac{\kappa_i-1}{L}\leq x_i<\frac{\kappa_i}{L}.
\]
Hence $x\in Q_\kappa$. 
There is only one index $\kappa$
such that $x \in Q_\kappa$
because the cubes
$Q_\nu$ form a disjoint partition of $Q$. Moreover,
\[
 3M<Lx_i<5M,
\]
and therefore
\begin{equation}\label{eq:260803-12}
 3M+1\leq\kappa_i\leq5M.
\end{equation}
For each $\nu=(\nu_1,\nu_2,\ldots,\nu_n)\in\{1,\ldots,L\}^n$, 
define the lattice distance by
\[
 d(\nu,\kappa)=\max\limits_{1\leq i\leq n}|\nu_i-\kappa_i|.
\]
We now construct explicitly a cube containing both $x$ and
$Q_\nu$. For the $i$-th coordinate, write
\[
 I_{\nu_i}
 =\left[\frac{\nu_i-1}{L},\frac{\nu_i}{L}\right),
 \qquad
 x_i\in I_{\kappa_i},
\]
and let
\[
 a_i=\min\left\{x_i,\frac{\nu_i-1}{L}\right\},
 \qquad
 b_i=\max\left\{x_i,\frac{\nu_i}{L}\right\}.
\]
Then both the point $x_i$ and the interval $I_{\nu_i}$ are contained in
$[a_i,b_i]$. Moreover,
\[
 b_i-a_i=
 \begin{cases}
 \dfrac{\nu_i}{L}-x_i,
 &\nu_i>\kappa_i (\iff \nu_i \ge \kappa_i+1),\\[6pt]
 \dfrac1L,
 &\nu_i=\kappa_i,\\[6pt]
 x_i-\dfrac{\nu_i-1}{L},
 &\nu_i<\kappa_i (\iff \nu_i \le \kappa_i-1).
 \end{cases}
\]
Recall that $\kappa_i=\lfloor L x_i\rfloor+1$.
Therefore,
$$
\frac{\kappa_i-1}{L}
\leq x_i<
\frac{\kappa_i}{L}.
$$ 
Each of these three cases gives
\begin{equation}\label{eq:260803-11}
 b_i-a_i
 \leq\frac{|\nu_i-\kappa_i|+1}{L}
 \leq\frac{d(\nu,\kappa)+1}{L}.
\end{equation}
Set
\[
 \delta=\frac{d(\nu,\kappa)+2}{L}
 \quad\text{and}\quad
 c_i=a_i-\frac{\delta-(b_i-a_i)}{2}=\frac{a_i+b_i-\delta}{2}.
\]
Estimate \eqref{eq:260803-11} implies $b_i-a_i<\delta$ and $c_i<a_i$, and hence
\[
 [a_i,b_i]\subset(c_i,c_i+\delta).
\]
Therefore the axis-parallel cube
\[
 R_{\nu,x}=\prod_{i=1}^n[c_i,c_i+\delta)
\]
contains $x$ and the whole cube $Q_\nu$, and its side length is
\[
 \ell(R_{\nu,x})=\delta
 =\frac{d(\nu,\kappa)+2}{L}.
\]
Using this cube in the definition of $\M$, we obtain
\begin{align}
 \M f_\nu(x)
 &\geq \frac1{|R_{\nu,x}|}\int_{R_{\nu,x}}\one_{Q_\nu}(y)\,dy
 =\frac{|Q_\nu|}{|R_{\nu,x}|} \notag\\
 &\geq \frac{L^{-n}}
 {\bigl[(d(\nu,\kappa)+2)/L\bigr]^n}
 =\bigl(d(\nu,\kappa)+2\bigr)^{-n}.
 \label{eq 20260702-11}
\end{align}

Let $s\in\N$.
We now count the indices
$\nu$
 at a fixed lattice distance $s$ from $\kappa$. 
Set
\[
 \Gamma_s(\kappa)
 =\{\nu\in\Z^n:\ d(\nu,\kappa)=s\}.
\]
If $1\leq s\leq M=L/8$, then
since
\[
 2M+1\leq\kappa_i-s\leq\nu_i\leq\kappa_i+s\leq6M<L
\] due to \eqref{eq:260803-12}, every $\nu\in\Gamma_s(\kappa)$ belongs to
$\{1,\ldots,L\}^n$. 
The number of integer vectors satisfying $d(\nu,\kappa)\leq s$ is
$(2s+1)^n$, because each coordinate can take one of the
$2s+1$ values
\[
 \kappa_i-s,\ \kappa_i-s+1,\ \ldots,\ \kappa_i+s.
\]
Similarly, the number satisfying $d(\nu,\kappa)\leq s-1$ is
$(2s-1)^n$. Therefore
\[
 \#\Gamma_s(\kappa)
 =(2s+1)^n-(2s-1)^n.
\]
The subtraction removes the vectors with $d(\nu,\kappa)\leq s-1$
from those with $d(\nu,\kappa)\leq s$, so the remaining vectors are
exactly the elements of $\Gamma_s(\kappa)$. Moreover,
\[
 (2s+1)^n-(2s-1)^n
 =\int_{2s-1}^{2s+1}nt^{n-1}\,dt
 \geq2n(2s-1)^{n-1}
 \geq2n s^{n-1}.
\]
Thus 
\begin{equation*}
 \#\Gamma_s(\kappa)\geq 2ns^{n-1}
 \qquad(1\leq s\leq M).
\end{equation*}

Remark that $s+2\leq3s$ for $s \in {\mathbb N}$.
If $q=1$, then \eqref{eq 20260702-11} and the preceding count give
\[
 \sum_{{\nu \in \{1,2,\ldots,L\}^n}}
 \M f_\nu(x)
 \geq\sum_{s=1}^{M}\sum_{\nu\in\Gamma_s(\kappa)}
 (s+2)^{-n}
 \geq \sum_{s=1}^{M}\frac{2n s^{n-1}}{(s+2)^n} 
 \geq \sum_{s=1}^{M}\frac{2n}{3^n s}.
 \]
Furthermore, using
$L=8M$, we have
 \[
 \sum_{s=1}^{M}\frac{2n}{3^n s} \ge \frac{2n}{3^n}\log (M+1) = \frac{2n}{4\cdot 3^n}\log (M+1)^4 \ge 
 \frac{2n}{4\cdot 3^n}\log 8M = \frac{2n}{4\cdot 3^n}\log L.
 \]
 Hence
 \[
 \sum_{{\nu \in \{1,2,\ldots,L\}^n}}
 \M f_\nu(x)
 \ge\frac{2n}{4\cdot 3^n}\log L.
 \]

Let $0<q<1$. Again using \eqref{eq 20260702-11}, we have
\[
 \sum_{\nu \in \{1,2,\ldots,L\}^n}
 (\M f_\nu(x))^q
 \geq\sum_{s=1}^{M}
 \frac{\sharp \Gamma_s(\kappa)}{(s+2)^{nq}} 
 \geq \sum_{s=1}^{M}\frac{2n s^{n-1}}{(s+2)^{nq}} 
 \geq \sum_{s=1}^{M}\frac{2n}{3^{n q}}s^{n(1-q)-1}.
 \]
 Hence 
 \[
 \sum_{\nu \in \{1,2,\ldots,L\}^n}
 (\M f_\nu(x))^q
 \geq \sum_{s=1}^{M}\frac{2n}{3^{n q}}s^{n(1-q)-1} 
 \geq c_{n,q}L^{n(1-q)}.
 \]
 The last inequality is the elementary power-sum estimate; for example, it
follows by summing only over the integers $s$ with $M/2\leq s\leq M$.

Consequently, for every $x\in E$,
\[
 \left( \sum_{\nu \in \{1,2,\ldots,L\}^n} (\M f_\nu(x))^q\right)^{\frac1q}
 \geq
 \begin{cases}
 c_n\log L,&q=1,\\
 c_{n,q}L^{n(1/q-1)},&0<q<1.
 \end{cases}
\]
The sets $E$ and $Q$ are bounded, so each is contained in a ball.
Definition~\ref{def 20260702-2}$(i)$,$(iii)$ therefore gives
$\one_E,\one_Q\in Y$. Hence \eqref{eq 20260702-10} and the lattice
property give
\[
 \left\|\left( \sum_{\nu \in \{1,2,\ldots,L\}^n}(\M f_\nu)^q\right)^{\frac1q}\right\|_Y
 \geq
 \begin{cases}
 c_n(\log L)\norm[Y]{\one_E},&q=1,\\
 c_{n,q}L^{n(1/q-1)}\norm[Y]{\one_E},&0<q<1,
 \end{cases}
\]
whereas the right-hand side of \eqref{eq 20260702-9}, for this family, is
$C\norm[Y]{\one_Q}$. The latter is independent of $L$, while either lower
bound tends to infinity as $L\to\infty$. This contradiction proves the
proposition.
\end{proof}

\begin{proposition}\label{prop 20260702-2}
Let $Y$ be a ball quasi-Banach function space. The following statements
are equivalent. The constant $C$ may be chosen to be the same in
$(i)$--$(iii)$.
\begin{enumerate}
 \item[$(i)$] There is a constant $C>0$ such that
 \[
 \norm[Y]{\M f}\leq C\norm[Y]{f}
 \qquad(f\in Y).
 \]
 \item[$(ii)$] There is a constant $C>0$ such that, for every $N\in\N$ and
 every finite family $\{f_j\}_{j=1}^N$ satisfying
 $\max\limits_{1\leq j\leq N}|f_j|\in Y$,
 \[
 \left\|\max\limits_{1\leq j\leq N}\M f_j\right\|_Y
 \leq C\left\|\max\limits_{1\leq j\leq N}|f_j|\right\|_Y.
 \]
 \item[$(iii)$] There is a constant $C>0$ such that
 \[
 \left\|\sup\limits_{j \in{\mathbb N}}\M f_j\right\|_Y
 \leq C\left\|\sup\limits_{j \in{\mathbb N}}|f_j|\right\|_Y
 \]
 whenever $\sup\limits_{j \in{\mathbb N}}|f_j|\in Y$.
\end{enumerate}
Moreover, even when $Y$ is a ball Banach function space, these
equivalent conditions need not imply that $\M$ is bounded on $Y'$.
\end{proposition}

\begin{proof}
For any finite or countable sequence,
\[
 \sup\limits_j\M f_j
 \leq \M\left(\sup\limits_j|f_j|\right)
\]
pointwise in $[0,\infty]$. Thus the scalar estimate in $(i)$, with
constant $C$, implies both $(ii)$ and $(iii)$ with the same constant.
Conversely, applying either $(ii)$ or $(iii)$ to a family with one nonzero
term gives $(i)$, again with the same constant. 

For the final assertion, take $Y=L^\infty(\R^n)$. Then $\M$ is
bounded on $Y$, whereas $Y'=L^1(\R^n)$ and $\M$ is not bounded on
$L^1(\R^n)$; see, for example, \cite[Chapter~I]{Stein1970}.

\end{proof}

\subsection{Dyadic grids and sparse families}
\label{subsection:Dyadic grids and sparse families}
A half-open cube is a set of the form
\[
 Q=\prod_{i=1}^n[a_i,a_i+\ell),
 \qquad a_i\in\R,\quad \ell>0.
\]
A collection $\D$ of half-open cubes is called a \emph{dyadic grid} if it
can be written as
\[
 \D=\bigcup\limits_{k\in\mathbb Z}\D_k
\]
and satisfies the following properties.
\begin{enumerate}
 \item[$(i)$] For every $k\in\mathbb Z$, $\D_k$ is a partition of
 $\R^n$ into half-open cubes of side length $2^{-k}$.
 \item[$(ii)$] Every $Q\in\D_k$ is contained in a unique cube
 $Q^{(1)}\in\D_{k-1}$ and is the disjoint union of exactly $2^n$
 cubes in $\D_{k+1}$.
 \item[$(iii)$] If $P,Q\in\D$, then $P\cap Q=\varnothing$, $P\subset Q$,
 or $Q\subset P$.
\end{enumerate}

A collection $\Sscr$ of cubes is called $\eta$-sparse, where
$0<\eta<1$, if there are measurable sets $E_Q\subset Q$,
$Q\in\Sscr$, such that
\begin{equation}\label{eq 20260702-14}
 |E_Q|\geq\eta|Q|
 \quad\text{and}\quad
 E_Q\cap E_P=\varnothing\quad(Q\neq P).
\end{equation}
If $\Sscr$ is contained in a dyadic grid, it is called a dyadic sparse
collection. For every finite collection $\Sscr$ of cubes and every
$0<q<\infty$, define
\begin{align}
 \T_{\Sscr,q}f
 &=\left(\sum_{Q\in\Sscr}\avg{|f|}_Q^q\one_Q\right)^{\frac1q},
 \label{eq 20260702-15}
 \end{align}
Every sparse collection of cubes in $\R^n$ is at most countable.
For an arbitrary sparse collection $\Sscr$, define
\[
 \A_{\Sscr}f(x)
 :=
 \sup\limits_{\substack{\Fscr\subset\Sscr\\ \Fscr\ {\rm finite}}}
 \sum_{Q\in\Fscr}\avg{|f|}_Q\one_Q(x)
 \in[0,\infty].
 \]
Since all the summands are nonnegative, this agrees with the
pointwise sum
\begin{equation}
 \label{eq 20260702-16}
 \A_{\Sscr}f(x)
 =
 \sum_{Q\in\Sscr}\avg{|f|}_Q\one_Q(x),
\end{equation}
independently of the enumeration of $\Sscr$.
Note that
the operator $\T_{\Sscr,q}$ will only be used below when $\Sscr$ is
finite.

For a dyadic grid $\D$ and a measurable function $f$, define the
dyadic maximal operator associated with $\D$ by
\[
 \M_{\D}f(x)
 :=
 \sup_{Q\in\D}\avg{|f|}_Q\one_Q(x)
 \in[0,\infty].
\]
If $\Fscr\subset\D$ is a nonempty finite subcollection, define the
corresponding finite dyadic maximal operator by
\[
 \M_{\Fscr}f(x)
 :=
 \max_{Q\in\Fscr}\avg{|f|}_Q\one_Q(x).
\]
We use the convention $\M_{\varnothing}f=0$.

Lemma \ref{lem 20260702-4} gives a finite sparse domination of the dyadic maximal
operator associated with a finite family of cubes.
Starting with the inclusion-maximal cubes in $\Fscr$, which are pairwise disjoint, we recursively select proper subcubes whose averages are more than twice the average of the previously selected cube.

At each stage,
we retain only those candidate subcubes that are not properly contained
in another candidate. This standard recursive selection produces a
finite $1/2$-sparse collection and controls every average occurring in
$\M_{\Fscr}f$; compare \cite{Lerner2013}. We include the proof because
Proposition~\ref{prop 20260704-1} below requires precisely this finite form
inside one fixed dyadic grid.

\begin{lemma}\label{lem 20260702-4}
Let $\D$ be a dyadic grid and let $\Fscr\subset\D$ be a finite
collection. For every locally integrable $f$, there is a
finite $1/2$-sparse collection $\Sscr\subset\D$ such that
\begin{equation}\label{eq 20260702-17}
 \M_{\Fscr}f\leq2\A_{\Sscr}f
 \qquad\text{almost everywhere}.
\end{equation}
\end{lemma}

\begin{proof}
If $\Fscr=\varnothing$, take $\Sscr=\varnothing$. Hence we may assume
that $\Fscr\neq\varnothing$.

Let $\mathcal C\subset\Fscr$ be a subcollection. A cube
$Q\in\mathcal C$ is called \emph{inclusion-maximal in $\mathcal C$} if
there is no cube $R\in\mathcal C$ such that $Q\subsetneq R$.

We construct $\Sscr$ recursively. First, select all cubes that are
inclusion-maximal in $\Fscr$. Suppose that $Q$ has already been selected
in the above procedure,
and consider the finite collection
\[
 \mathcal B(Q)
 =\left\{P\in\Fscr:\ P\subsetneq Q,
 \ \avg{|f|}_P>2\avg{|f|}_Q\right\}.
\]
Assume that $\mathcal B(Q) \ne \emptyset$.
Select all cubes that are inclusion-maximal in $\mathcal B(Q)$. We call
them the \emph{selected children} of $Q$. Here ``children'' refers only
to this recursive selection: these cubes need not be the $2^n$
immediate dyadic children of $Q$. Apply the same rule to every newly
selected cube. Each selected child is a proper subcube of its parent,
and $\Fscr$ is finite. Hence the procedure stops after finitely many
steps. Let $\Sscr$ be the collection of all selected cubes, and for
$Q\in\Sscr$ write $\operatorname{ch}_{\Sscr}(Q)$ for the selected
children of $Q$.
Notice that a cube $R$ belongs to $\mathscr S$ if and only if either
$R$ is inclusion-maximal in $\mathscr F$, or there exists a cube
$Q\in\mathscr S$ such that $R$ is inclusion-maximal in $\mathcal B(Q)$.

We first prove that $\Sscr$ is $1/2$-sparse. The cubes in
$\operatorname{ch}_{\Sscr}(Q)$ are pairwise disjoint due to maximality.
If $\avg{|f|}_Q>0$, the definition of the
children and their pairwise disjointness give
\[
 2\avg{|f|}_Q
 \sum_{P\in\operatorname{ch}_{\Sscr}(Q)}|P|
 <\sum_{P\in\operatorname{ch}_{\Sscr}(Q)}\int_P|f|
 \leq\int_Q|f|
 =\avg{|f|}_Q|Q|.
\]
Consequently,
\[
 \left|\bigcup\limits_{P\in\operatorname{ch}_{\Sscr}(Q)}P\right|
 \leq\frac{|Q|}{2}.
\]
If $\avg{|f|}_Q=0$, then $|f|=0$ almost everywhere on $Q$. Hence every
subcube $P\subset Q$ has $\avg{|f|}_P=0$, and no child of $Q$ is
selected. In either case, the set
\[
 E_Q
 =Q\setminus
 \bigcup\limits_{P\in\operatorname{ch}_{\Sscr}(Q)}P
\]
satisfies $|E_Q|\geq|Q|/2$.

It remains to check that the sets $E_Q$, $Q\in\Sscr$, are pairwise
disjoint. Let $Q,R\in\Sscr$ and $Q\neq R$. If $Q\cap R=\varnothing$,
then clearly $E_Q\cap E_R=\varnothing$. Otherwise, the dyadic-grid
property implies that one cube contains the other. Suppose that
$R\subsetneq Q$ by symmetry. By the recursive construction, there is a finite chain
of selected cubes
\[
 Q=Q_0\supsetneq Q_1\supsetneq\cdots\supsetneq Q_m=R,
\]
where $Q_{i+1}$ is a child of $Q_i$ for every $i$. In particular,
$Q_1\in\operatorname{ch}_{\Sscr}(Q)$ and $R\subseteq Q_1$. Setting
$P=Q_1$, we obtain
\[
 E_R\subset R\subseteq P,
 \qquad
 E_Q\subset Q\setminus P,
\]
so $E_Q\cap E_R=\varnothing$. Thus $\Sscr$ is $1/2$-sparse.

We now prove the pointwise estimate
\eqref{eq 20260702-17}. Fix $R\in\Fscr$. Since $\Fscr$ is
finite
and $\D$ is nesting, there is a unique inclusion-maximal cube $Q_0$ of $\Fscr$ that contains
$R$, and this cube was selected at the first stage. 

Starting from $Q_0$, proceed as follows. If $R$ is contained in a child
of the current selected cube, replace the current cube by that child.
Such a child is unique because the children are pairwise disjoint. Each
replacement gives a proper subcube, and $\Sscr$ is finite, so after
finitely many steps we reach a selected cube $Q$ such that
\[
 R\subseteq Q
\]
and that
\begin{equation}\label{eq:260803-13}
 R\not\subseteq P
 \end{equation}
for every $P\in\operatorname{ch}_{\Sscr}(Q)$. 
We show that this terminal cube $Q$
satisfies
\[
 \avg{|f|}_R\leq2\avg{|f|}_Q.
\]
If $R=Q$, this is immediate. Suppose that $R\subsetneq Q$ and that this
inequality is false. Then
\[
 \mathcal G
 =\left\{P\in\Fscr:\ R\subseteq P\subsetneq Q,
 \ \avg{|f|}_P>2\avg{|f|}_Q\right\}
\]
is nonempty because $R\in\mathcal G$. Choose an inclusion-maximal cube
$P$ in $\mathcal G$. 
There are two possibilities, and both lead to a contradiction.
\begin{itemize}
\item If $P$ is not inclusion-maximal in $\mathcal B(Q)$, then there
exists $P'\in\mathcal B(Q)$ such that $P\subsetneq P'$. Since
$R\subseteq P\subsetneq P'\subsetneq Q$, we have $P'\in\mathcal G$,
which contradicts the inclusion-maximality of $P$ in $\mathcal G$.

\item If $P$ is inclusion-maximal in $\mathcal B(Q)$, then $P$ is a
selected child of $Q$. Since $R\subseteq P$, this contradicts
\eqref{eq:260803-13}. 
\end{itemize}
Therefore
\[
 \avg{|f|}_R\leq2\avg{|f|}_Q.
\]

For $x\in R$, the cube $Q$ also contains $x$, and the term
$\avg{|f|}_Q\one_Q(x)$ occurs in $\A_{\Sscr}f(x)$. Hence
\[
 \avg{|f|}_R\one_R(x)
 \leq2\avg{|f|}_Q\one_Q(x)
 \leq2\A_{\Sscr}f(x).
\]
The same inequality is trivial when $x\notin R$
since the left-hand side is zero. Taking the maximum over
$R\in\Fscr$ proves \eqref{eq 20260702-17}.
\end{proof}

We next explain how finitely many dyadic maximal operators control the
ordinary maximal operator. A standard construction of adjacent dyadic
grids provides $\D^1,\ldots,\D^{3^n}$ with
the following property: for every axis-parallel cube $Q$, there are an
index $t\in\{1,\ldots,3^n\}$ and a cube $R\in\D^t$ such that
\begin{equation}\label{eq 20260702-18}
 Q\subset R,
 \qquad
 |R|\leq 6^n |Q|.
\end{equation}
Thus an arbitrary cube can be enlarged to a dyadic cube from one of the
fixed grids, while its measure increases by at most the factor
$ 6^n $. If $x\in Q$ and $R$ is chosen as above, then
\[
 \avg{|f|}_Q
 \leq\frac{|R|}{|Q|}\avg{|f|}_R
 \leq 6^n \M_{\D^t}f(x).
\]
Taking the supremum over all cubes $Q$ containing $x$ gives
\begin{equation}\label{eq 20260702-19}
 \M f(x)
 \leq 6^n \max\limits_{1\leq t\leq3^n}\M_{\D^t}f(x)
 \leq 6^n \sum_{t=1}^{3^n}\M_{\D^t}f(x),
\end{equation}
See \cite{Lerner2013} for this construction.

\section{Sparse bounds obtained from the vector-valued estimate}\label{sec 20260703-1}

\subsection{A vector-valued sparse estimate}

The proof of the next proposition uses the sets $E_Q$ from the
definition of sparseness in two ways. For a fixed nonnegative function
$f$ and each $Q\in\Sscr$, consider the auxiliary function
$\avg{f}_Q\one_{E_Q}$. Since the sets $E_Q$ are pairwise disjoint, the
$\ell^q$-sum
\[
\left(\sum_{Q \in \Sscr}|\avg{f}_Q\one_{E_Q}|^q\right)^{\frac1q}
\] 
of these auxiliary functions is controlled by $\M f$. On the
other hand, the estimate $|E_Q|\geq\eta|Q|$ gives
\[
 \M\bigl(\avg{f}_Q\one_{E_Q}\bigr)(x)
 \geq\eta\avg{f}_Q
 \qquad(x\in Q).
\]
These two observations give the estimate in the next proposition.

\begin{proposition}\label{prop 20260703-1}
Let $X$ be a ball Banach function space, let $1<q<\infty$, and assume
that \eqref{eq 20260702-7} holds on $X$ with a constant $C>0$
independent of $N$. If $\Sscr$ is a finite $\eta$-sparse collection,
then
\begin{equation}\label{eq 20260703-1}
 \norm[X]{\T_{\Sscr,q}f}
 \leq \eta^{-1}C^2\norm[X]{f}
 \qquad(f\in X).
\end{equation}
The constant is independent of $\Sscr$ and of the measurable sets
$E_Q$ used in \eqref{eq 20260702-14}.
\end{proposition}

\begin{proof}
It is enough to consider the case $f\geq0$. 
Choose measurable sets $E_Q\subset Q$, $Q\in\Sscr$, satisfying
\eqref{eq 20260702-14}, and define
\begin{equation}\label{eq 20260703-2}
 h_Q=\avg{f}_Q\one_{E_Q}.
\end{equation}
Each $h_Q$ belongs to $X$ by Lemma~\ref{lem 20260702-2}$(i)$. Since the sets
$E_Q$ are pairwise disjoint,
\[
 \left(\sum_{Q\in\Sscr}|h_Q|^q\right)^{\frac1q}
 =\sum_{Q\in\Sscr}\avg{f}_Q\one_{E_Q}.
\]
If $x\in E_Q$, then $x\in Q$, and hence
$\avg{f}_Q\leq\M f(x)$. Therefore
\begin{equation}\label{eq 20260703-3}
 \left(\sum_{Q\in\Sscr}|h_Q|^q\right)^{\frac1q}
 \leq\M f.
\end{equation}

For $x\in Q$,
\[
 \M h_Q(x)
 \geq\avg{h_Q}_Q
 =\avg{f}_Q\frac{|E_Q|}{|Q|}
 \geq\eta\avg{f}_Q.
\]
Consequently,
\begin{equation}\label{eq 20260703-4}
 \eta\T_{\Sscr,q}f=\eta\left(\sum_{Q\in\Sscr}\avg{|f|}_Q^q\one_Q\right)^{\frac1q}
 \leq\left(\sum_{Q\in\Sscr}(\M h_Q)^q\right)^{\frac1q}.
\end{equation}
Recall that, by taking $N=1$ in \eqref{eq 20260702-7},
we obtain
\[
 \norm[X]{\M f}\leq C\norm[X]{f}.
\]
Therefore, by \eqref{eq 20260703-4} and
\eqref{eq 20260703-3}, 
\begin{align*}
 \eta\norm[X]{\T_{\Sscr,q}f}
 \leq C
 \left\|\left(\sum_{Q\in\Sscr}|h_Q|^q\right)^{\frac1q}\right\|_X
 \leq C\norm[X]{\M f}
 \leq C^2\norm[X]{f}.
\end{align*}
This is the desired result.
\end{proof}

\subsection{Binomial iteration}

Let $\{a_j\}_{j=1}^N$ be a finite family of nonnegative numbers, and let
$q>1$. Then
\[
 \left(\sum_{j=1}^Na_j^q\right)^{\frac1q}
 \leq\sum_{j=1}^Na_j,
\]
but the reverse inequality requires a constant dependent on $N$. Thus a
bound for $\T_{\Sscr,q}$ does not immediately give a bound for
$\A_{\Sscr}$. The next theorem shows that repeated application of
$\T_{\Sscr,q}$ produces binomial coefficients large enough to control
the sum defining $\A_{\Sscr}$. Since $\Sscr$ is finite, if
$g\in L^1_{\mathrm{loc}}(\R^n)$, then every average occurring in
$\T_{\Sscr,q}g$ is finite, and $\T_{\Sscr,q}g$ is a bounded
measurable function supported on $\bigcup\limits_{Q\in\Sscr}Q$. Hence
$\T_{\Sscr,q}g\in L^1_{\mathrm{loc}}(\R^n)$, so every iteration below
is well defined for locally integrable functions.

\begin{theorem}\label{thm 20260703-1}
Let $1<q<\infty$, let $m\in\N$ satisfy $m>q$, and let
$\Sscr$ be a finite subcollection of a dyadic grid. Put
\begin{equation}\label{eq 20260703-5}
 \Gamma_{q,m}
 =\left(
 \sum_{\ell=1}^{\infty}
 \binom{\ell+m-2}{m-1}^{-1/(q-1)}
 \right)^{1/q'}.
\end{equation}
Then $\Gamma_{q,m}<\infty$, and
\begin{equation}\label{eq 20260703-6}
 \A_{\Sscr}f(x)
 \leq\Gamma_{q,m}\,\T_{\Sscr,q}^{\,m}f(x)
\end{equation}
for every locally integrable $f$ and almost every $x$. Here
$\T_{\Sscr,q}^{\,m}$ denotes the $m$-fold composition of
$\T_{\Sscr,q}$.
\end{theorem}

\begin{proof}
It is enough to consider $f\geq0$. Put
\[
 \Omega_{\Sscr}=\bigcup\limits_{Q\in\Sscr}Q.
\]
If $x\notin\Omega_{\Sscr}$, then $\one_Q(x)=0$ for every
$Q\in\Sscr$. Hence, by \eqref{eq 20260702-15} and
\eqref{eq 20260702-16},
\[
 \A_{\Sscr}f(x)=0
 \qquad\text{and}\qquad
 \T_{\Sscr,q}g(x)=0
\]
for every function $g$ for which $\T_{\Sscr,q}g$ is defined. It follows
that $\T_{\Sscr,q}^{\,m}f(x)=0$, so \eqref{eq 20260703-6} holds at
$x$. Therefore, it remains to prove \eqref{eq 20260703-6} for
$x\in\Omega_{\Sscr}$.

\emph{Step 1. A lower bound for the iterations.}
For each $x\in\Omega_{\Sscr}$, the cubes of $\Sscr$ containing $x$ form
a nonempty finite chain,
since $\Sscr$ is a finite set. Write this chain as
\[
 Q_1\supsetneq Q_2\supsetneq\cdots\supsetneq Q_N,
 \qquad
 a_j=\avg{f}_{Q_j}.
\]
Thus $N\geq1$. We prove by induction on $k\geq1$ that, for every
$x\in\Omega_{\Sscr}$,
\begin{equation}\label{eq 20260703-7}
 \bigl(\T_{\Sscr,q}^{\,k}f(x)\bigr)^q
 \geq
 \sum_{j=1}^N
 \binom{N-j+k-1}{k-1}a_j^q,
\end{equation}
where $N$, $Q_1,\ldots,Q_N$, and $a_1,\ldots,a_N$ are determined by
$x$ as above.

For $k=1$, only the cubes $Q_1,\ldots,Q_N$ contribute to
\eqref{eq 20260702-15}. Therefore,
\[
 \bigl(\T_{\Sscr,q}f(x)\bigr)^q
 =\sum_{j=1}^Na_j^q.
\]
Since $\binom{N-j}{0}=1$ for $1\leq j\leq N$, this proves
\eqref{eq 20260703-7} when $k=1$.

Assume that \eqref{eq 20260703-7} holds for an integer $k\geq1$ at every
point of $\Omega_{\Sscr}$. Let $1\leq s\leq N$ and $y\in Q_s$.
The cubes
$Q_1,\ldots,Q_s$ all contain $y$. Write all cubes of $\Sscr$ containing
$y$ in decreasing order as
\begin{equation}\label{eq:260803-34}
 R_1\supsetneq R_2\supsetneq\cdots\supsetneq R_L.
\end{equation}
We claim that
\[
 R_j=Q_j\qquad(1\leq j\leq s).
\]
Indeed, let $R\in\Sscr$ contain $y$ and suppose that $R$ is not strictly
contained in $Q_s$. Since $R$ and $Q_s$ are dyadic cubes in the same
grid and intersect at $y$, one of them must contain the other. By our
assumption on $R$, we therefore have $Q_s\subset R$. Since
$x\in Q_s$, it follows that $R$ also contains $x$. Hence $R$ is one of
the cubes $Q_1,\ldots,Q_s$. Conversely, each of the cubes
$Q_1,\ldots,Q_s$ contains $Q_s$ and therefore contains $y$.
Thus, $Q_1,\ldots,Q_s$ are the only cubes in $\Sscr$ that contain $y$
and are not strictly contained in $Q_s$.
Since \eqref{eq:260803-34} holds, we have $Q_j=R_j$
for all $j=1,2,\ldots,s$. This proves the claim. In particular,
$L\geq s$, and any further cubes $R_{s+1},\ldots,R_L$, if they occur,
are strictly contained in $Q_s$.
Put $b_j=\avg{f}_{R_j}$ for $1\leq j\leq L$. By the claim,
$b_j=a_j$ for $1\leq j\leq s$. Applying the induction hypothesis at the
point $y$ to the chain $R_1,\ldots,R_L$, and then keeping only its first
$s$ nonnegative terms, gives
\begin{align*}
 \bigl(\T_{\Sscr,q}^{\,k}f(y)\bigr)^q
 &\geq
 \sum_{j=1}^L
 \binom{L-j+k-1}{k-1}b_j^q\\
 &\geq
 \sum_{j=1}^s
 \binom{L-j+k-1}{k-1}a_j^q \\
 &\geq
 \sum_{j=1}^s
 \binom{s-j+k-1}{k-1}a_j^q.
\end{align*}
For the last inequality, we used $L\geq s$ and the fact that
$u\mapsto\binom{u}{k-1}$ is nondecreasing for integers $u\geq k-1$.

Set
\[
 B_s=\sum_{j=1}^s
 \binom{s-j+k-1}{k-1}a_j^q.
\]
The preceding estimate holds for every $y\in Q_s$, and hence
\[
 \T_{\Sscr,q}^{\,k}f(y)\geq B_s^{\frac1q}
 \qquad(y\in Q_s).
\]
Taking the average over $Q_s$ gives
\[
 \avg{\T_{\Sscr,q}^{\,k}f}_{Q_s}
 \geq B_s^{\frac1q},
 \qquad
 \left(\avg{\T_{\Sscr,q}^{\,k}f}_{Q_s}\right)^q
 \geq B_s.
\]
The cubes of $\Sscr$ that contain $x$ are exactly
$Q_1,\ldots,Q_N$. Therefore, by the definition of
$\T_{\Sscr,q}$,
\begin{align*}
 \bigl(\T_{\Sscr,q}^{\,k+1}f(x)\bigr)^q
 &=\sum_{s=1}^N
 \left(\avg{\T_{\Sscr,q}^{\,k}f}_{Q_s}\right)^q\\
 &\geq\sum_{s=1}^NB_s\\
 &=\sum_{s=1}^N\sum_{j=1}^s
 \binom{s-j+k-1}{k-1}a_j^q\\
 &=\sum_{j=1}^N
 \left(\sum_{s=j}^N
 \binom{s-j+k-1}{k-1}\right)a_j^q\\
 &=\sum_{j=1}^N
 \binom{N-j+k}{k}a_j^q.
\end{align*}
In the last equality, we set $u=s-j$ and used the binomial identity
\[
 \sum_{s=j}^N\binom{s-j+k-1}{k-1}
 =\sum_{u=0}^{N-j}\binom{u+k-1}{k-1}
 =\binom{N-j+k}{k}.
\]
The final expression is the right-hand side of \eqref{eq 20260703-7}
with $k+1$ in place of $k$.
This proves \eqref{eq 20260703-7} for $k+1$.

\emph{Step 2. Comparison with $\A_{\Sscr}$.}
Apply \eqref{eq 20260703-7} with $k=m$ and set
\[
 w_j=\binom{N-j+m-1}{m-1}.
\]
H\"older's inequality gives
\begin{align*}
 \sum_{j=1}^Na_j
 &\leq
 \left(\sum_{j=1}^Nw_ja_j^q\right)^{\frac1q}
 \left(\sum_{j=1}^Nw_j^{-1/(q-1)}\right)^{1/q'}
 \leq
 \Gamma_{q,m}\,\T_{\Sscr,q}^{\,m}f(x).
\end{align*}
The left-hand side is $\A_{\Sscr}f(x)$. To finish the proof, 
it remains to verify the convergence in
\eqref{eq 20260703-5}. Since $m>q>1$, one has $m\geq2$. For every
integer $\ell\geq1$,
\begin{align*}
 \binom{\ell+m-2}{m-1}
 &=\frac{1}{(m-1)!}\prod_{i=0}^{m-2}(\ell+i).
\end{align*}
Each factor in the product is at least $\ell$. Also,
$\ell+i\leq\ell+m-2\leq(m-1)\ell$ for $0\leq i\leq m-2$.
Therefore
\begin{equation}\label{eq 20260703-7a}
 \frac{\ell^{m-1}}{(m-1)!}
 \leq
 \binom{\ell+m-2}{m-1}
 \leq
 \frac{(m-1)^{m-1}}{(m-1)!}\,\ell^{m-1}
 \qquad(\ell\geq1).
\end{equation}
Thus the binomial coefficient in \eqref{eq 20260703-7a} is comparable to
$\ell^{m-1}$ for every $\ell\geq1$, with constants depending only on $m$.
In particular, the summand in \eqref{eq 20260703-5} is bounded above by a
constant depending only on $m$ and $q$ times
\[
 \ell^{-(m-1)/(q-1)}.
\]
Since $m>q$, one has $(m-1)/(q-1)>1$. The corresponding $p$-series
converges, and hence $\Gamma_{q,m}<\infty$.
\end{proof}

\begin{corollary}\label{cor 20260703-1}
Let $X$ be a ball Banach function space, let $1<q<\infty$, and assume
that \eqref{eq 20260702-7} holds on $X$ with a constant $C>0$
independent of $N$. Let $m\in\N$ satisfy $m>q$. If $\Sscr$ is a finite
$\eta$-sparse subcollection of a dyadic grid, then
\begin{equation}\label{eq 20260703-8}
 \norm[X]{\A_{\Sscr}f}
 \leq
 \Gamma_{q,m}\bigl(\eta^{-1}C^2\bigr)^m\norm[X]{f}
 \qquad(f\in X).
\end{equation}
In particular, the constant is independent of the dyadic grid and of
$\Sscr$.
\end{corollary}

\begin{proof}
By Proposition~\ref{prop 20260703-1},
\[
 \norm[X]{\T_{\Sscr,q}f}
 \leq K\norm[X]{f},
 \qquad
 K=\eta^{-1}C^2.
\]
Applying this estimate $m$ times yields
\[
 \norm[X]{\T_{\Sscr,q}^{\,m}f}\leq K^m\norm[X]{f}.
\]
Now use \eqref{eq 20260703-6}.
\end{proof}

\section{A finite dyadic sparse criterion and the proof of Theorem~\ref{thm 20260701-1}}
\label{sec 20260704-1}

\subsection{The finite dyadic sparse criterion}

Lorist and Nieraeth proved the following observation
\cite[Lemma~3.4]{LoristNieraethCompact2024}.
\begin{lemma}
\label{lemm 20260810-1-noi}
For a Banach function space $X$, the two
bounds
\[
 \M:X\to X
 \qquad\text{and}\qquad
 \M:X'\to X'
\]
are equivalent to the uniform estimate
\[
 \sup\limits_{\Sscr\,\text{is }1/2\text{-sparse}}
 \norm[X\to X]{\A_{\Sscr}}<\infty.
\]
\end{lemma}
The same equivalence is
also stated in
\cite[Theorem~1.2, conditions~$(iii)$ and~$(iv)$]{Nieraeth2026}. The spaces
considered here satisfy the hypotheses of these results: the Fatou
property is part of Definition~\ref{def 20260702-1}, and the saturation
property follows from Lemma~\ref{lem 20260702-2}$(iii)$.

Section~\ref{sec 20260703-1} gives a uniform estimate only for finite
$1/2$-sparse subcollections of dyadic grids. For the reader's
convenience, we explain why this restricted class is sufficient. The
implication from the two maximal-operator bounds to the finite dyadic
sparse estimate follows from
\cite[Lemma~3.4]{LoristNieraethCompact2024}. The converse follows directly
from Lemma~\ref{lem 20260702-4}, the Fatou properties of $X$ and $X'$, and
the finite family of adjacent dyadic grids in
\eqref{eq 20260702-18}--\eqref{eq 20260702-19}. We record the precise
form needed below.

\begin{proposition}\label{prop 20260704-1}
Let $X$ be a ball Banach function space. The following statements are
equivalent.
\begin{enumerate}
 \item[$(i)$] The operators $\M:X\to X$ and $\M:X'\to X'$ are bounded.
 \item[$(ii)$] There is a constant $C_S$ such that
 \[
 \norm[X]{\A_{\Sscr}f}\leq C_S\norm[X]{f}
 \qquad(f\in X)
 \]
 for every finite $1/2$-sparse subcollection $\Sscr$ of every dyadic
 grid.
\end{enumerate}
\end{proposition}

\begin{proof}
Assume $(i)$. By the sparse characterization in
\cite[Lemma~3.4]{LoristNieraethCompact2024}, the operators
$\A_{\Sscr}$ are uniformly bounded on $X$ for all sparse collections.
In particular, the estimate in $(ii)$ holds for every finite
$1/2$-sparse subcollection of every dyadic grid.

Conversely, assume $(ii)$. We first show that the same sparse estimate
holds on $X'$. Let $\Sscr$ be a finite $1/2$-sparse subcollection of a
dyadic grid and let $g\in X'$. For every $f\in X$, the finiteness of
$\Sscr$ gives
\begin{align*}
 \int_{\R^n}|f|\,\A_{\Sscr}g
 =\sum_{Q\in\Sscr}\avg{|g|}_Q\int_Q|f|
 =\sum_{Q\in\Sscr}\avg{|f|}_Q\int_Q|g|
 =\int_{\R^n}\A_{\Sscr}f\,|g|.
\end{align*}
Hence, by \eqref{eq 20260702-3} and $(ii)$,
\[
 \int_{\R^n}|f|\,\A_{\Sscr}g
 \leq \norm[X]{\A_{\Sscr}f}\norm[X']{g}
 \leq C_S\norm[X]{f}\norm[X']{g}.
\]
Taking the supremum over all $f\in X$ with $\norm[X]{f}\leq1$ yields
\begin{equation}\label{eq 20260704-1}
 \norm[X']{\A_{\Sscr}g}\leq C_S\norm[X']{g}.
\end{equation}

Fix a dyadic grid $\D$ and a finite subcollection $\Fscr\subset\D$.
Let $Z$ denote either $X$ or $X'$, and let $f\in Z$. By
Lemma~\ref{lem 20260702-2}$(ii)$, the function $f$ is locally integrable.
Lemma~\ref{lem 20260702-4} gives a finite $1/2$-sparse collection
$\Sscr\subset\D$ such that
\[
 \M_{\Fscr}f\leq2\A_{\Sscr}f
 \qquad\text{almost everywhere}.
\]
Using $(ii)$ when $Z=X$ and \eqref{eq 20260704-1} when $Z=X'$, we obtain
\begin{equation}\label{eq 20260704-2}
 \norm[Z]{\M_{\Fscr}f}
 \leq2C_S\norm[Z]{f}.
\end{equation}

Each generation $\D_m$ is countable: its cubes have pairwise disjoint
interiors, and each such interior contains a rational point. Hence
$\D=\bigcup\limits_{m\in\Z}\D_m$ is countable. Choose increasing finite
subcollections $\Fscr_k\subset\D$ such that
$\bigcup\limits_{k=1}^\infty\Fscr_k=\D$. Then
$\M_{\Fscr_k}f\uparrow\M_{\D}f$ pointwise. The Fatou property of $X$
and Lemma~\ref{lem 20260702-associate-fatou} for $X'$ therefore give, from
\eqref{eq 20260704-2},
\[
 \norm[Z]{\M_{\D}f}
 \leq2C_S\norm[Z]{f}
 \qquad(Z=X\text{ or }X').
\]
Finally, let $\D^1,\ldots,\D^{3^n}$ be the adjacent dyadic grids in
\eqref{eq 20260702-18}. By \eqref{eq 20260702-19},
\[
 \norm[Z]{\M f}
 \leq 6^n \sum_{t=1}^{3^n}\norm[Z]{\M_{\D^t}f}
 \leq2\cdot 18^n C_S\norm[Z]{f}
 \qquad(Z=X\text{ or }X').
\]
Thus $\M$ is bounded on both $X$ and $X'$, which proves $(i)$.
\end{proof}

\subsection{Proof of Theorem~\ref{thm 20260701-1}}

We first prove parts $(ii)$, $(iii)$, and $(iv)$. Part $(ii)$ is
Proposition~\ref{prop 20260702-1}. For part $(iii)$, fix
$1<q<\infty$. Lemma~\ref{lem 20260702-3} shows that an estimate valid
for every $N\in\N$ extends, with the same constant, to every sequence
$\{f_j\}_{j=1}^{\infty}$ for which
$\bigl(\sum_{j=1}^{\infty}|f_j|^q\bigr)^{\frac1q}$ belongs to $X$. The converse
follows by adding zero terms. This proves part $(iii)$. Part $(iv)$ is
Proposition~\ref{prop 20260702-2}.

It remains to prove part $(i)$.

\emph{Step 1. $(b)$ $\Longrightarrow$ $(a)$.}
Assume $(b)$. Then there are $q\in(1,\infty)$ and a finite constant $C$
such that \eqref{eq 20260701-3} holds for every $N$, with $C$
independent of $N$. 
If we take $N=1$, then we see that $\M$ is bounded on $X$.
By Corollary~\ref{cor 20260703-1}, the operators
$\A_{\Sscr}$ are uniformly bounded on $X$ over all finite
$1/2$-sparse subcollections of all dyadic grids.
Proposition~\ref{prop 20260704-1} then shows
$\M$ is bounded on $X'$.
Thus $(a)$
holds.

\emph{Step 2. $(a)$ $\Longrightarrow$ $(c)$.}
Assume $(a)$, and fix $q\in(1,\infty)$. Choose $p_0\in(1,\infty)$. For
every $N\in\N$ and every finite family $\{f_j\}_{j=1}^N$ of measurable
functions, put
\[
 G=\left(\sum_{j=1}^N(\M f_j)^q\right)^{\frac1q},
 \qquad
 F=\left(\sum_{j=1}^N|f_j|^q\right)^{\frac1q}.
\]
The weighted inequality in Theorem~\ref{thm 20260702-1} verifies the
hypothesis of Theorem~\ref{thm 20260702-2}, with a bound independent of
$N$. Applying that theorem to the collection of all pairs $(G,F)$ gives
\eqref{eq 20260701-3} on $X$. Since $q\in(1,\infty)$ was arbitrary,
$(c)$ follows.

\emph{Step 3. $(c)$ $\Longrightarrow$ $(b)$.}
This implication is immediate.

This completes the proof of Theorem~\ref{thm 20260701-1}.

\section{Proof of Theorem~\ref{thm 20260701-2}}\label{sec 20260705-1}

Theorem~\ref{thm 20260701-2} is proved by passing from the ball quasi-Banach function space
$Y$ to the rescaled space $X=Y^r$. We first recall the properties of
this rescaling that are needed in the proof.

\subsection{Power rescaling}

We recall the $r$-concavification of a ball quasi-Banach function space $Y$.
\begin{definition}\label{def 20260705-1}
Let $Y$ be a ball quasi-Banach function space and let $r>0$. Define
\begin{equation}\label{eq 20260705-1}
 Y^r
 =\{h:\ |h|^{1/r}\in Y\},
 \qquad
 \norm[Y^r]{h}
 =\norm[Y]{|h|^{1/r}}^r.
\end{equation}
The space $Y^r$, equipped with the quasi-norm in \eqref{eq 20260705-1}, is called the
$r$-concavification of $Y$.
\end{definition}
This is the power notation used in \cite{Nieraeth2023}.

The next lemma records the elementary properties of the power space
$Y^r$ that are used below.

\begin{lemma}\label{lem 20260705-1}
Let $Y$ be a ball quasi-Banach function space and let $r>0$. Then
$Y^r$, equipped with the quasi-norm in \eqref{eq 20260705-1}, is a ball
quasi-Banach function space. For every nonnegative $H\in Y$,
\begin{equation}\label{eq 20260705-2}
 \norm[Y]{H}^r=\norm[Y^r]{H^r}.
\end{equation}
\end{lemma}

\begin{proof}
By Definition~\ref{def 20260702-2}$(i)$,$(iii)$, $Y$ has the saturation
property. Hence the standard concavification result in
\cite[p.~12]{Nieraeth2023} applies and shows that $Y^r$, equipped with
the quasi-norm in \eqref{eq 20260705-1}, is a quasi-Banach function
space. Since $\one_B^{1/r}=\one_B$ for every ball $B$, the
ball-indicator property is inherited from $Y$. Finally,
\eqref{eq 20260705-2} follows directly from
\eqref{eq 20260705-1}.
\end{proof}

We also recall the $r$-convexity of constant one.
\begin{definition}\label{def 20260705-2}
Let $Y$ be a ball quasi-Banach function space and let $r>0$. Following
\cite[Definition~2.7]{Nieraeth2023}, we say that $Y$ is
\emph{$r$-convex with constant one} if
\begin{equation}\label{eq 20260705-4}
 \left\|\bigl(|f|^r+|g|^r\bigr)^{1/r}\right\|_Y
 \leq\bigl(\norm[Y]{f}^r+\norm[Y]{g}^r\bigr)^{1/r}
 \qquad(f,g\in Y).
\end{equation}
\end{definition}

By the standard concavification criterion, $Y$ is $r$-convex with
constant one if and only if the quasi-norm in
\eqref{eq 20260705-1} is a norm on $Y^r$; see
\cite[p.~12]{Nieraeth2023}. Combining this criterion with
Lemma~\ref{lem 20260705-1}, we obtain Lemma \ref{lem 20260705-2} below.
At this stage, the continuous embedding of $Y^r$ into
$L^1_{\mathrm{loc}}(\mathbb R^n)$ has not yet been established.
This will follow from Lemma~\ref{lem 20260705-3} below whenever the maximal
operator is bounded on $Y^r$.

\begin{lemma}\label{lem 20260705-2}
Let $Y$ be an $r$-convex ball quasi-Banach function space with
constant one, and put $X=Y^r$ with the quasi-norm in
\eqref{eq 20260705-1}. Then this quasi-norm is actually a norm, and $X$ is a
Banach lattice of measurable functions with the Fatou property and containing
the indicator of every ball.
\end{lemma}

\begin{proof}
By the standard concavification criterion
\cite[p.~12]{Nieraeth2023}, the quasi-norm in
\eqref{eq 20260705-1} is a norm on $Y^r$ because $Y$ is $r$-convex.
The remaining assertions follow from
Lemma~\ref{lem 20260705-1}.
\end{proof}

\subsection{Local integrability and powered operators}

Comparing Definitions \ref{def 20260702-1} and \ref{def 20260702-2}, we see that, for a ball quasi-Banach function space $Z$ to be a ball Banach function space, it is necessary, in addition to the triangle inequality for the norm, that every element of $Z$ belong to $L^1_{\mathrm{loc}}$.
The following lemma provides a sufficient condition for this property.

\begin{lemma}\label{lem 20260705-3}
Let $Z$ be a ball quasi-Banach function space whose given quasi-norm is a
norm. If $\M$ is bounded on $Z$, 
then $Z$ is continuously embedded into
$L^1(B)$ for every ball $B$. Consequently, $Z$ is a ball Banach
function space. 
\end{lemma}

Remark that $Z$ is a normed function space
(see Definition \ref{def 20260702-1}).
Therefore, its K\"{o}the dual is defined by \eqref{eq 20260702-2}.
\begin{proof}
By Definition~\ref{def 20260702-2}$(i)$,$(iii)$, $Z$ has the saturation
property. We may therefore apply \cite[Lemma~2.26]{Nieraeth2023} to
the basis of axis-parallel cubes. It follows that $\one_Q\in Z'$ for
every axis-parallel cube $Q$. Let $B$ be a ball and choose such a cube
$Q_B$ containing $B$. Then the generalized H\"older inequality gives
\[
 \int_B|f(x)|\,dx
 \leq\int_{Q_B}|f(x)|\,dx
 \leq\norm[Z]{f}\norm[Z']{\one_{Q_B}}
 \qquad(f\in Z).
\]
Thus $Z$ is continuously embedded into $L^1(B)$. Since the given
quasi-norm is a norm, all the remaining axioms in
Definition~\ref{def 20260702-1} are already part of the assumptions on
$Z$.
\end{proof}

It is useful to compare the preceding lemma with an earlier related result.
In \cite[Lemma 2.4]{MastyloSawano2026},
we assumed that $\M$ is bounded on $Z'$.
In both cases, we use the existence of a non-trivial element in $Z'$.

Let $r>0$, let $\Sscr$ be a finite collection of cubes.
We recall again that we defined
$\M_r$ and $\A_{\Sscr,r}$ in \eqref{eq 20260701-6}. 

The assumption $f\in L^r_{\mathrm{loc}}(\R^n)$ means precisely that
$|f|^r\in L^1_{\mathrm{loc}}(\R^n)$, so every cube average in
\eqref{eq 20260701-6} is finite. The function $\M_r f$ may nevertheless
take the value $+\infty$, because the supremum defining $\M$ is taken
over all cubes. In the proof below, any one of the
conditions in part $(i)$ of Theorem~\ref{thm 20260701-2} will imply
$Y\subset L^r_{\mathrm{loc}}(\R^n)$. For conditions $(b)$ and $(d)$ in
that part, this is proved by applying the assumed estimate to bounded compactly
supported truncations.

\subsection{Proof of Theorem~\ref{thm 20260701-2}}

Put $X=Y^r$. By Lemma~\ref{lem 20260705-2}, $X$ is a Banach lattice with the Fatou
property and contains the indicator of every ball.

\emph{Step 1. The case $0<q<\infty$.}
Let $0<q<\infty$, put $a=q/r$, and let
$f_j\in L^r_{\mathrm{loc}}(\R^n)$. Set $h_j=|f_j|^r$. From
\eqref{eq 20260705-1} and \eqref{eq 20260701-6},
\begin{align}
 \left\|\left(\sum_{j=1}^N(\M_r f_j)^q\right)^{\frac1q}\right\|_Y^r
 &=\left\|\left(\sum_{j=1}^N(\M h_j)^a\right)^{1/a}\right\|_X,
 \label{eq 20260705-6}\\
 \left\|\left(\sum_{j=1}^N|f_j|^q\right)^{\frac1q}\right\|_Y^r
 &=\left\|\left(\sum_{j=1}^N|h_j|^a\right)^{1/a}\right\|_X.
 \label{eq 20260705-7}
\end{align}
Thus an estimate on $X$ with exponent $a$ and constant $D$ gives
\eqref{eq 20260701-7} with constant $D^{1/r}$. Conversely,
\eqref{eq 20260701-7} with constant $C$ gives the corresponding estimate
on $X$ with constant $C^r$ for finite families of nonnegative locally
integrable functions for which the quantity on the right is finite.

Assume that \eqref{eq 20260701-7} holds with a constant $C$. We now
remove the local-integrability restriction. Let $h_1,\ldots,h_N$ be
nonnegative measurable functions such that
\[
 H=\left(\sum_{j=1}^N h_j^a\right)^{1/a}\in X.
\]
For $k\in\N$, define
\[
 h_{j,k}=\min\{h_j,k\}\one_{B(0,k)},
 \qquad
 f_{j,k}=h_{j,k}^{1/r}.
\]
Then $f_{j,k}\in Y\cap L^r_{\mathrm{loc}}(\R^n)$ and
$\bigl(\sum_{j=1}^Nh_{j,k}^a\bigr)^{1/a}\leq H$. Hence
\eqref{eq 20260701-7} and
\eqref{eq 20260705-6}--\eqref{eq 20260705-7} give
\[
 \left\|\left(\sum_{j=1}^N(\M h_{j,k})^a\right)^{1/a}\right\|_X
 \leq C^r\norm[X]{H}.
\]
As $k\to\infty$, one has $h_{j,k}\uparrow h_j$ and
$\M h_{j,k}\uparrow\M h_j$. The Fatou property of $X$ therefore gives
\[
 \left\|\left(\sum_{j=1}^N(\M h_j)^a\right)^{1/a}\right\|_X
 \leq C^r
 \left\|\left(\sum_{j=1}^N h_j^a\right)^{1/a}\right\|_X.
\]
Taking $N=1$ shows that $\M$ is bounded on $X$. 
By
Lemma~\ref{lem 20260705-3}, $X$ is a ball Banach function space, and
hence $Y\subset L^r_{\mathrm{loc}}(\R^n)$. If $q\leq r$, then
$a\leq1$, and the last vector-valued estimate contradicts
Proposition~\ref{prop 20260702-1}. Thus $q>r$. Theorem~\ref{thm 20260701-1}
now gives $\M$ is bounded on $X'$,
so condition $(b)$ in part $(i)$ implies
condition $(a)$. The same argument proves part $(ii)$.

Conversely, assume condition $(a)$ in part $(i)$.
Lemma~\ref{lem 20260705-3} shows that $X$ is a
ball Banach function space and therefore that
$Y\subset L^r_{\mathrm{loc}}(\R^n)$. For every $q>r$, one has
$a=q/r>1$. Theorem~\ref{thm 20260701-1} gives
\eqref{eq 20260701-3} on $X$ with exponent $a$ and some finite constant
$D$. The identities \eqref{eq 20260705-6}--\eqref{eq 20260705-7} then
give \eqref{eq 20260701-7} with constant $D^{1/r}$. Hence, within
part $(i)$, $(a)$ implies $(c)$, while $(c)$ implies $(b)$ immediately.

Under the hypothesis of part $(iii)$, part $(ii)$ gives
$Y\subset L^r_{\mathrm{loc}}(\R^n)$. Since $q>r$ by assumption, one has
$a=q/r>1$. For every sequence considered in part $(iii)$,
\[
 |f_j|\leq\left(\sum_{k=1}^\infty|f_k|^q\right)^{\frac1q}\in Y,
\]
so each $f_j$ belongs to $L^r_{\mathrm{loc}}(\R^n)$ and $\M_r f_j$ is
well defined. The argument above gives the finite vector-valued estimate
on $X$ with exponent $a=q/r$ and constant $C^r$.
Theorem~\ref{thm 20260701-1}$(iii)$, together with
\eqref{eq 20260705-6}--\eqref{eq 20260705-7} and increasing partial
sums, therefore gives the estimate in part $(iii)$ with constant $C$.

\emph{Step 2. The sparse condition.}
For every finite sparse collection $\Sscr$ and every
$f\in Y\cap L^r_{\mathrm{loc}}(\R^n)$,
\begin{equation}\label{eq 20260705-9}
 \norm[Y]{\A_{\Sscr,r}f}^r
 =\norm[X]{\A_{\Sscr}(|f|^r)}.
\end{equation}

Assume condition $(a)$ in part $(i)$. Lemma~\ref{lem 20260705-3}
shows that $X$ is a ball Banach
function space, so $Y\subset L^r_{\mathrm{loc}}(\R^n)$.
Proposition~\ref{prop 20260704-1} gives a constant $D>0$, independent of
$\Sscr$, such that
\[
 \norm[X]{\A_{\Sscr}h}\leq D\norm[X]{h}
 \qquad(h\in X).
\]
By \eqref{eq 20260705-9},
\[
 \norm[Y]{\A_{\Sscr,r}f}
 \leq D^{1/r}\norm[Y]{f}
 \qquad(f\in Y).
\]
Hence $(a)$ implies $(d)$ in part $(i)$.

Conversely, assume condition $(d)$ in part $(i)$, and let
$C_{\mathrm{sp}}$ be the constant in
that condition. We first verify local integrability of the elements of
$X$. Let $h\in X$, and for $k\in\N$ define
\[
 h_k=\min\{|h|,k\}\one_{B(0,k)},
 \qquad
 f_k=h_k^{1/r}.
\]
Then $f_k\in Y\cap L^r_{\mathrm{loc}}(\R^n)$. Let $B$ be a ball,
choose an axis-parallel cube $Q_0$ containing $B$, and use
\eqref{eq 20260702-18} to choose a dyadic cube $Q$ such that
$Q_0\subset Q$. The one-cube collection $\{Q\}$ is $1/2$-sparse.
By condition $(d)$ and \eqref{eq 20260705-9},
\[
 \avg{h_k}_Q\norm[X]{\one_Q}
 =\norm[X]{\A_{\{Q\}}h_k}
 =\norm[Y]{\A_{\{Q\},r}f_k}^r
 \leq C_{\mathrm{sp}}^r\norm[Y]{f_k}^r
 \leq C_{\mathrm{sp}}^r\norm[X]{h}.
\]
Letting $k\to\infty$ and using the monotone convergence theorem gives
\[
 \int_B|h(x)|\,dx
 \leq
 \frac{|Q|C_{\mathrm{sp}}^r}{\norm[X]{\one_Q}}\norm[X]{h}.
\]
Thus $X$ is a ball Banach function space and
$Y\subset L^r_{\mathrm{loc}}(\R^n)$. For $h\in X$, apply condition $(d)$
to $f=|h|^{1/r}$. Equation \eqref{eq 20260705-9} gives
\[
 \norm[X]{\A_{\Sscr}h}
 \leq C_{\mathrm{sp}}^r\norm[X]{h}
\]
for every finite $1/2$-sparse subcollection $\Sscr$ of every dyadic
grid. Proposition~\ref{prop 20260704-1} now gives
$\M$ is bounded on both $X$ and $X'$.
Hence $(d)$ implies $(a)$ in part $(i)$.

\emph{Step 3. The supremum estimate.}
Assume that $Y\subset L^r_{\mathrm{loc}}(\R^n)$, as in part $(iv)$ of
the theorem. Let $\{f_j\}$ be a finite or countable family
such that
\[
 F:=\sup\limits_j|f_j|\in Y.
\]
For every $j$, one has $|f_j|^r\leq F^r$. Hence the monotonicity of
$\M$ gives
\[
 \M_r f_j
 =\bigl[\M(|f_j|^r)\bigr]^{1/r}
 \leq\bigl[\M(F^r)\bigr]^{1/r}
 =\M_r F.
\]
Therefore
\begin{equation}\label{eq 20260705-10}
 \sup\limits_j\M_r f_j\leq\M_r F
 \qquad\text{pointwise}.
\end{equation}
If $\M_r:Y\to Y$ is bounded and
$\norm[Y]{\M_r f}\leq C\norm[Y]{f}$ for every $f\in Y$, then
\eqref{eq 20260705-10} and the lattice property give
\[
 \norm[Y]{\sup\limits_j\M_r f_j}
 \leq\norm[Y]{\M_r F}
 \leq C\norm[Y]{F}
 =C\norm[Y]{\sup\limits_j|f_j|}.
\]
Conversely, if the supremum estimate holds with a constant $C$, then a
family with one nonzero term gives
\[
 \norm[Y]{\M_r f}\leq C\norm[Y]{f}
 \qquad(f\in Y).
\]
Thus the same constant works in the scalar estimate and in the supremum
estimate. This proves part $(iv)$.

Taking $Y=L^\infty(\R^n)$ shows that boundedness of $\M_r$ on $Y$
need not imply boundedness of $\M$ on $X'=L^1(\R^n)$.

This completes the proof of Theorem~\ref{thm 20260701-2}.

\section{Applications to local maximal operators}
\label{sec 20260812-2-noi}

We consider the local version of Theorems~\ref{thm 20260701-1}
and~\ref{thm 20260701-2}. 
Let $f$ be a measurable function on $\R^n$.
For $\rho\in(0,\infty]$, define
\[
 \Mloc_\rho f(x)
 =\sup\limits_{\substack{Q\ni x\\ \ell(Q)\leq\rho}}
  \avg{|f|}_Q,
\]
where the supremum is taken over all axis-parallel cubes in $\R^n$.
We call $\rho$ the truncation parameter; the case $\rho=\infty$
corresponds to the untruncated operator. When
$\rho=\infty$, that is, when the restriction on $\ell(Q)$ is removed, we have
\[
 \Mloc_\infty=\M.
\]
Furthermore, $\Mloc_1$ is the usual
normalized local maximal operator \cite{Rychkov2001}. 

We organize Section~\ref{sec 20260812-2-noi} as follows.
Section~\ref{subsection:The characterization} states the main results of Section~\ref{sec 20260812-2-noi}.
Section~\ref{subsection:Comparison of finite truncation parameters} compares various finite truncations with different parameters.
Section \ref{subsection:Weighted inequalities and extrapolation at a fixed truncation parameter} collects
inequalities related to $\Mloc_a$.
Section~\ref{section:q<1,infinity} deals with the endpoint cases.
Section~\ref{sec 20260812-3-noi} establishes the sparse criterion for a finite truncation parameter.
Section~\ref{subsection:From a vector-valued bound to sparse averaging} derives a sparse estimate from the vector-valued inequality, and Section~\ref{subsec 0813-1-noi} proves Theorem~\ref{thm 20260812-1-noi}.

\subsection{The characterization}
\label{subsection:The characterization}
Let $X$ be a ball Banach function space and let $X'$ be its K\"othe
associate space. For $0<q<\infty$ and $\rho\in(0,\infty)$ consider
\begin{equation}\label{eq 20260812-3-noi}
 \left\|
 \left(\sum_{j=1}^N(\Mloc_\rho f_j)^q\right)^{1/q}
 \right\|_X
 \leq C
 \left\|
 \left(\sum_{j=1}^N|f_j|^q\right)^{1/q}
 \right\|_X,
\end{equation}
where $C$ is independent of $N\in\N$ and of the finite sequence $\{f_j\}_{j=1}^N$. 

We have the corresponding statement to Theorem \ref{thm 20260701-1}.
\begin{theorem}\label{thm 20260812-1-noi}
Let $\rho\in(0,\infty)$. Let $X$ be a ball Banach function space on
$\R^n$, and let $X'$ be its K\"othe associate space.
\begin{enumerate}
\item[$(i)$] The following statements are equivalent.
\begin{enumerate}
\item[$(a)$] $\Mloc_\rho:X\to X$ and $\Mloc_\rho:X'\to X'$ are bounded.

\item[$(b)$] There exist $q\in(1,\infty)$ and a constant $C>0$ such that, for
 every $N\in\N$ and every finite family $\{f_j\}_{j=1}^N$ of
 measurable functions, the finite vector-valued inequality
\begin{equation}\label{eq 20260812-6-noi}
 \left\|
 \left(\sum_{j=1}^N(\Mloc_\rho f_j)^q\right)^{\frac1q}
 \right\|_X
 \leq C
 \left\|
 \left(\sum_{j=1}^N|f_j|^q\right)^{\frac1q}
 \right\|_X
\end{equation}
 holds whenever the quantity on the right is finite. The constant $C$ is
 independent of $N$ and of the family $\{f_j\}_{j=1}^N$.

\item[$(c)$] For every $q\in(1,\infty)$, there is a constant $C>0$ such that
 \eqref{eq 20260812-6-noi} holds for every $N\in\N$, with $C$ independent
 of $N$ and of the family $\{f_j\}_{j=1}^N$.

\end{enumerate}

\item[$(ii)$] The parameter $q\in(0,\infty)$ must satisfy $q>1$ if there exists
 a finite constant $C$ satisfying the inequality
 \eqref{eq 20260812-6-noi} for every $N\in\N$ and every finite family
 $\{f_j\}_{j=1}^N$ of measurable functions. More precisely, for
 $0<q\leq1$ such an estimate fails on every ball quasi-Banach function
 space on $\R^n$.

\item[$(iii)$] Fix $1<q<\infty$. Suppose that the inequality
 \eqref{eq 20260812-6-noi}, with this value of $q$ and constant $C$, holds
 for every $N\in\N$. Then, for every sequence
 $\{f_j\}_{j=1}^{\infty}$ satisfying
 \[
 \left(\sum_{j=1}^{\infty}|f_j|^q\right)^{\frac1q}\in X,
 \]
 one has
\begin{equation}\label{eq 20260812-7-noi}
 \left\|
 \left(\sum_{j=1}^{\infty}(\Mloc_\rho f_j)^q\right)^{\frac1q}
 \right\|_X
 \leq C
 \left\|
 \left(\sum_{j=1}^{\infty}|f_j|^q\right)^{\frac1q}
 \right\|_X.
\end{equation}
 Conversely, this estimate for all sequences implies
 \eqref{eq 20260812-6-noi} by taking $f_j=0$ for $j>N$, with the same
 constant $C$.

\item[$(iv)$] The following statements are equivalent. The constant $C$ may be
 chosen to be the same in $(a)$--$(c)$.
\begin{enumerate}
\item[$(a)$] There is a constant $C>0$ such that
 \[
 \|\Mloc_\rho f\|_X\leq C\|f\|_X
 \qquad(f\in X).
 \]

\item[$(b)$] There is a constant $C>0$ such that, for every $N\in\N$ and every
 finite family $\{f_j\}_{j=1}^N$ satisfying
 $\max\limits_{1\leq j\leq N}|f_j|\in X$,
 \[
 \left\|\max\limits_{1\leq j\leq N}\Mloc_\rho f_j\right\|_X
 \leq C
 \left\|\max\limits_{1\leq j\leq N}|f_j|\right\|_X.
 \]

\item[$(c)$] There is a constant $C>0$ such that, for every sequence
 $\{f_j\}_{j=1}^{\infty}$ with $\sup\limits_{j \in{\mathbb N}}|f_j|\in X$,
 \[
 \left\|\sup\limits_{j \in{\mathbb N}}\Mloc_\rho f_j\right\|_X
 \leq C
 \left\|\sup\limits_{j \in{\mathbb N}}|f_j|\right\|_X.
 \]
\end{enumerate}
These equivalent conditions need not imply that $\Mloc_\rho$ is bounded on $X'$.
\end{enumerate}
\end{theorem}

We have the corresponding statement to Theorem \ref{thm 20260701-2}.
To formulate the powered result, let $Y$ be a ball quasi-Banach function
space and let $r>0$.
Recall that $Y^r$ is defined by \eqref{eq 20260701-5}
and that $Y$ is $r$-convex with constant one if \eqref{eq 20260705-4} holds.
Under this assumption, $X=Y^r$ is a Banach lattice with the Fatou property.

For $\rho\in(0,\infty]$, define
\begin{equation}\label{eq 20260812-10-noi}
 \Mloc_{\rho,r}f
 =\bigl[\Mloc_\rho(|f|^r)\bigr]^{1/r}.
\end{equation}

\begin{theorem}
\label{thm 20260812-2-noi}
Let $\rho\in(0,\infty)$. Let $Y$ be an $r$-convex ball quasi-Banach
function space with convexity constant one for some $r>0$, and put
$X=Y^r$ with the norm in \eqref{eq 20260701-5}.
\begin{enumerate}
\item[$(i)$] The following statements are equivalent.
\begin{enumerate}
\item[$(a)$] $\Mloc_\rho:X\to X$ and $\Mloc_\rho:X'\to X'$ are bounded.

\item[$(b)$] There exist $q>r$ and a constant $C>0$ such that, for every
 $N\in\N$ and every finite family
 $\{f_j\}_{j=1}^N\subset L^r_{\mathrm{loc}}(\R^n)$,
\begin{equation}\label{eq 20260812-11-noi}
 \left\|
 \left(\sum_{j=1}^N({{\Mloc_{\rho,r}}}f_j)^q\right)^{\frac1q}
 \right\|_Y
 \leq C
 \left\|
 \left(\sum_{j=1}^N|f_j|^q\right)^{\frac1q}
 \right\|_Y
\end{equation}
 whenever the quantity on the right is finite. The constant $C$ is
 independent of $N$ and of the family $\{f_j\}_{j=1}^N$.

\item[$(c)$] For every $q>r$, there is a constant $C>0$ such that
 \eqref{eq 20260812-11-noi} holds for every $N\in\N$ and every finite
 family $\{f_j\}_{j=1}^N\subset L^r_{\mathrm{loc}}(\R^n)$ for which the
 quantity on the right is finite, with $C$ independent of $N$ and of the
 family.

\end{enumerate}
Each of these conditions implies $Y\subset L^r_{\mathrm{loc}}(\R^n)$.
Moreover, for $q>r$, a constant $C$ works in
\eqref{eq 20260812-11-noi} if and only if $C^r$ works in
\eqref{eq 20260812-6-noi} on $X$ with $q/r$ in place of $q$.

\item[$(ii)$] Let $q>0$. Suppose that there is a finite constant $C$ such that
 \eqref{eq 20260812-11-noi} holds for every $N\in\N$ and every finite
 family $\{f_j\}_{j=1}^N\subset L^r_{\mathrm{loc}}(\R^n)$ for which the
 quantity on the right is finite, with $C$ independent of $N$ and of the
 family. Then necessarily $q>r$ and
 $Y\subset L^r_{\mathrm{loc}}(\R^n)$. In particular, no estimate in
 \eqref{eq 20260812-11-noi} can hold for every $N\in\N$ when
 $0<q\leq r$.

\item[$(iii)$] Fix $q>r$ and suppose that \eqref{eq 20260812-11-noi}, with constant
 $C$, holds for every $N\in\N$. Then, for every sequence
 $\{f_j\}_{j=1}^{\infty}$ satisfying
 $\bigl(\sum_{j=1}^{\infty}|f_j|^q\bigr)^{\frac1q}\in Y$, one has
\begin{equation}\label{eq 20260812-13-noi}
 \left\|
 \left(\sum_{j=1}^{\infty}({{\Mloc_{\rho,r}}}f_j)^q\right)^{\frac1q}
 \right\|_Y
 \leq C
 \left\|
 \left(\sum_{j=1}^{\infty}|f_j|^q\right)^{\frac1q}
 \right\|_Y.
\end{equation}
 Conversely, this estimate for all sequences implies
 \eqref{eq 20260812-11-noi} by taking $f_j=0$ for $j>N$, with the same
 constant $C$.

\item[$(iv)$] Suppose that $Y\subset L^r_{\mathrm{loc}}(\R^n)$. The following
 statements are equivalent. The constant $C$ may be chosen to be the same
 in $(a)$--$(c)$.
\begin{enumerate}
\item[$(a)$] There is a constant $C>0$ such that
 \[
 \|{{\Mloc_{\rho,r}}}f\|_Y\leq C\|f\|_Y
 \qquad(f\in Y).
 \]

\item[$(b)$] There is a constant $C>0$ such that, for every $N\in\N$ and every
 finite family $\{f_j\}_{j=1}^N$ satisfying
 $\max\limits_{1\leq j\leq N}|f_j|\in Y$,
 \[
 \left\|\max\limits_{1\leq j\leq N}{{\Mloc_{\rho,r}}}f_j\right\|_Y
 \leq C
 \left\|\max\limits_{1\leq j\leq N}|f_j|\right\|_Y.
 \]

\item[$(c)$] There is a constant $C>0$ such that, for every sequence
 $\{f_j\}_{j=1}^{\infty}$ with $\sup\limits_{j \in{\mathbb N}}|f_j|\in Y$,
 \[
 \left\|\sup\limits_{j \in{\mathbb N}}{{\Mloc_{\rho,r}}}f_j\right\|_Y
 \leq C
 \left\|\sup\limits_{j \in{\mathbb N}}|f_j|\right\|_Y.
 \]
\end{enumerate}
\end{enumerate}
\end{theorem}

\begin{remark}
\label{rem 20260812-1-noi}
Let $a,b\in(0,\infty)$. Replacing the truncation parameter $a$ by $b$
does not change any of the boundedness or vector-valued properties in
Theorems~\ref{thm 20260812-1-noi} and~\ref{thm 20260812-2-noi}. More
precisely, each statement involving $\Mloc_a$ holds if and only if the
corresponding statement involving $\Mloc_b$ holds; in the powered case,
the same assertion holds with ${{\Mloc_{a,r}}}$ and ${{\Mloc_{b,r}}}$. The
constants may depend on $a$ and $b$. This equivalence concerns only finite
truncation parameters and does not compare these operators with the global
Hardy--Littlewood maximal operator $\M=\Mloc_\infty$. Consequently,
boundedness for one, and hence every, finite truncation parameter does not
by itself imply boundedness of $\M$.
\end{remark}

The power rescaling underlying Theorem~\ref{thm 20260812-2-noi} is
particularly simple. If $h_j=|f_j|^r$ and $a=q/r$, then
\begin{align}
 \left\|
 \left(\sum_{j=1}^N({{\Mloc_{\rho,r}}}f_j)^q\right)^{1/q}
 \right\|_Y^r
 &=
 \left\|
 \left(\sum_{j=1}^N(\Mloc_\rho h_j)^a\right)^{1/a}
 \right\|_X,
 \label{eq 20260812-14-noi}\\
 \left\|
 \left(\sum_{j=1}^N|f_j|^q\right)^{1/q}
 \right\|_Y^r
 &=
 \left\|
 \left(\sum_{j=1}^Nh_j^a\right)^{1/a}
 \right\|_X.
 \label{eq 20260812-15-noi}
\end{align}
Thus the threshold $q>r$ is exactly the threshold $a>1$ in
Theorem~\ref{thm 20260812-1-noi}.

\subsection{Comparison of finite truncation parameters}
\label{subsection:Comparison of finite truncation parameters}

\begin{lemma}[Pointwise comparison of finite truncation parameters]
\label{lem 20260812-2-noi}
Let $0<a<b<\infty$. There exist $m\in\N$ and
$C=C(n,b/a)>0$ such that
\begin{equation}\label{eq 20260812-18-noi}
 \Mloc_b f(x)\leq C(\Mloc_a)^m f(x)
\end{equation}
for every measurable function $f$ and every $x\in\R^n$, with both sides
understood as extended nonnegative quantities.
\end{lemma}

\begin{proof}
Set $R=b/a$. In
\cite[proof of Proposition~3.3]{NogayamaSawano2021}, the following
pointwise estimate is established: there exist $m=m(R)\in\N$ and
$C_{n,R}>0$ such that
\[
 \Mloc_R h(z)\leq C_{n,R}(\Mloc_1)^m h(z)
\]
for every measurable function $h$ and every $z\in\R^n$.

Let $g(y)=f(ay)$. A change of variables in the defining averages gives
\[
 \Mloc_R g(x/a)=\Mloc_b f(x),
 \qquad
 \Mloc_1 g(y)=\Mloc_a f(ay).
\]
Iterating the second identity yields
\[
 (\Mloc_1)^m g(x/a)=(\Mloc_a)^m f(x).
\]
Therefore,
\[
 \Mloc_b f(x)
 =\Mloc_R g(x/a)
 \leq C_{n,R}(\Mloc_1)^m g(x/a)
 =C_{n,R}(\Mloc_a)^m f(x),
\]
which proves \eqref{eq 20260812-18-noi}.
\end{proof}

If we apply \eqref{eq 20260812-18-noi}, we see that the value
$a$ in $\Mloc_a$ is not so important as long as $a>0$.
\begin{corollary}
\label{cor 20260812-1-noi}
Let $Z$ be a ball quasi-Banach function space continuously embedded into
$L^1_{\mathrm{loc}}(\R^n)$, and let $a,b\in(0,\infty)$.
\begin{enumerate}
\item[$(i)$] The operator $\Mloc_a$ is bounded on $Z$ if and only if
$\Mloc_b$ is bounded on $Z$.
\item[$(ii)$] Fix $0<q<\infty$. The following statements are equivalent.
\begin{enumerate}
\item[$(a)$] There is $C_a>0$ such that
\[
 \left\|
 \left(\sum_{j=1}^N(\Mloc_a f_j)^q\right)^{1/q}
 \right\|_Z
 \leq C_a
 \left\|
 \left(\sum_{j=1}^N|f_j|^q\right)^{1/q}
 \right\|_Z
\]
for every $N\in\N$ and every sequence $\{f_j\}_{j=1}^N$ for which the
quantity on the right is finite, with $C_a$ independent of $N$ and of the
sequence.
\item[$(b)$] There is $C_b>0$ such that
\[
 \left\|
 \left(\sum_{j=1}^N(\Mloc_b f_j)^q\right)^{1/q}
 \right\|_Z
 \leq C_b
 \left\|
 \left(\sum_{j=1}^N|f_j|^q\right)^{1/q}
 \right\|_Z
\]
for every $N\in\N$ and every sequence $\{f_j\}_{j=1}^N$ for which the
quantity on the right is finite, with $C_b$ independent of $N$ and of the
sequence.
\end{enumerate}
\end{enumerate}
\end{corollary}

\begin{proof}
It is enough to consider $a<b$. If $\Mloc_a$ is bounded on $Z$, then
\eqref{eq 20260812-18-noi} and iteration give boundedness of $\Mloc_b$.
Likewise, repeated application of the corresponding vector-valued
estimate for $\Mloc_a$ gives
\[
 \left\|
 \left(\sum_{j=1}^N[(\Mloc_a)^mf_j]^q\right)^{1/q}
 \right\|_Z
 \leq C_a^m
 \left\|
 \left(\sum_{j=1}^N|f_j|^q\right)^{1/q}
 \right\|_Z.
\]
Combining this with \eqref{eq 20260812-18-noi} proves the estimate for
$\Mloc_b$. The reverse implications follow from
$\Mloc_a f\leq\Mloc_b f$.
\end{proof}

We slightly refine Lemma \ref{lem 20260705-3}.
\begin{lemma}
\label{lem 20260812-3-noi}
Let $Z$ be a ball quasi-Banach function space whose given quasi-norm is a
norm. If $\Mloc_\rho$ is bounded on $Z$ for some
$\rho\in(0,\infty)$, then $Z$ is continuously embedded into $L^1(B)$ for
every ball $B$. Consequently, $Z$ is a ball Banach function space.
\end{lemma}

\begin{proof}
This follows by 
covering $B$ with cubes of volume $\rho^n$ 
and applying the argument in the proof of Lemma \ref{lem 20260705-3}.
\end{proof}

\subsection{Weighted inequalities and extrapolation at a fixed truncation parameter}
\label{subsection:Weighted inequalities and extrapolation at a fixed truncation parameter}
A weight is a locally integrable function $w:\R^n\to(0,\infty)$ that
is positive and finite almost everywhere.
For $\rho\in(0,\infty]$ and $1<p<\infty$, define
\begin{equation}\label{eq 20260812-19-noi}
 [w]_{A_p(\rho)}
 =\sup\limits_{\ell(Q)\leq\rho}
 \avg{w}_Q
 \left(\avg{w^{-1/(p-1)}}_Q\right)^{p-1}.
\end{equation}
We write $w\in A_p(\rho)$ when this quantity is finite. Thus
$A_p(1)=A_p^{\mathrm{loc}}$ and $A_p(\infty)=A_p$. For $p=1$, set
\begin{equation}\label{eq 20260812-20-noi}
 [w]_{A_1(\rho)}
 =\mathop{\mathrm{ess\,sup}}_{x\in\R^n}
 \frac{\Mloc_\rho w(x)}{w(x)}.
\end{equation}

\begin{proposition}[Weighted vector-valued inequality with truncation parameter $\rho$]
\label{prop 20260812-1-noi}
Let $\rho\in(0,\infty)$ and $1<p,q<\infty$. There is an increasing
function $\mathcal N_{n,p,q}:[1,\infty)\to(0,\infty)$ such that, for every
$w\in A_p(\rho)$, every $N\in\N$, and every finite family for which the
quantity on the right is finite,
\begin{equation}\label{eq 20260812-21-noi}
 \left\|
 \left(\sum_{j=1}^N(\Mloc_\rho f_j)^q\right)^{1/q}
 \right\|_{L^p(w)}
 \leq
 \mathcal N_{n,p,q}([w]_{A_p(\rho)})
 \left\|
 \left(\sum_{j=1}^N|f_j|^q\right)^{1/q}
 \right\|_{L^p(w)}.
\end{equation}
The function $\mathcal N_{n,p,q}$ is independent of $N$.
\end{proposition}
For $\rho=\infty$, \eqref{eq 20260812-21-noi} is the weighted
vector-valued maximal inequality of Andersen and John
\cite{AndersenJohn1981}, with the same increasing-envelope convention;
see Theorem \ref{thm 20260702-1}. 
\begin{proof}
We remark that, when $\rho=1$,
\eqref{eq 20260812-21-noi} is precisely
\cite[Lemma~2.11]{Rychkov2001}. We now explain how to extend this
inequality to arbitrary values of $\rho<\infty$.
For $0<\rho<\infty$, set
\[
 \widetilde f_j(y)=f_j(\rho y),
 \qquad
 \widetilde w(y)=w(\rho y).
\]
A change of variables gives
\[
 \Mloc_1\widetilde f_j(y)=\Mloc_\rho f_j(\rho y),
 \qquad
 [\widetilde w]_{A_p(1)}=[w]_{A_p(\rho)},
\]
and, for every measurable $H$,
\[
 \|H(\rho\,\cdot)\|_{L^p(\widetilde w)}
 =\rho^{-n/p}\|H\|_{L^p(w)}.
\]
Hence Rychkov's estimate for $\Mloc_1$
\cite[Lemma~2.11]{Rychkov2001} yields
\eqref{eq 20260812-21-noi} with truncation parameter $\rho$. The same local result also appears in
\cite[Proposition~1.9]{NogayamaSawano2021}. The constants in
these arguments are uniform when the corresponding $A_p$ characteristic
is bounded; taking the increasing envelope of this dependence gives
$\mathcal N_{n,p,q}$. 
\end{proof}

We point out that the local counterpart to Theorem \ref{thm 20260702-2} is obtained
in \cite[Theorem 3.1]{CMM22}.
Proposition \ref{prop 20260812-2-noi} is known as \cite[Theorem 3.1]{CMM22}.
To keep track of the constant in \eqref{eq 20260812-22-noi},
we supply the proof.

\begin{proposition}[Extrapolation at a fixed truncation parameter]
\label{prop 20260812-2-noi}
Let $\rho\in(0,\infty)$. Let $X$ be a ball Banach function space such that
$\Mloc_\rho$ is bounded on both $X$ and $X'$. Fix
$p_0\in(1,\infty)$ and let
$\mathcal N:[1,\infty)\to(0,\infty)$ be increasing. Let $\mathcal F$ be a
family of pairs $(F,G)$ of nonnegative measurable functions. Assume that,
for every $w\in A_{p_0}(\rho)$ and every $(F,G)\in\mathcal F$ with
$G\in L^{p_0}(w)$, one has $F\in L^{p_0}(w)$ and
\begin{equation}\label{eq 20260812-22-noi}
 \|F\|_{L^{p_0}(w)}
 \leq\mathcal N([w]_{A_{p_0}(\rho)})\|G\|_{L^{p_0}(w)}.
\end{equation}
Then there is $C=C(X,\rho,p_0,\mathcal N)>0$ such that
\begin{equation}\label{eq 20260812-23-noi}
 \|F\|_X\leq C\|G\|_X
\end{equation}
for every $(F,G)\in\mathcal F$ with $G\in X$.
\end{proposition}

When $\rho=\infty$, the assertion follows from the extrapolation theorem
of Nieraeth \cite[Theorem~4.7 and Remark~4.8]{Nieraeth2023}. 
See Theorem \ref{thm 20260702-2}.
\begin{proof}
Assume $0<\rho<\infty$ and set
\[
 A_X=\max\{1,\|\Mloc_\rho\|_{X\to X}\},
 \qquad
 A_{X'}=\max\{1,\|\Mloc_\rho\|_{X'\to X'}\}.
\]
The half-open unit cubes $k+[0,1)^n$, $k\in\mathbb Z^n$, are pairwise
disjoint and their union is $\R^n$. Since this collection is countable,
choose a sequence $\{Q_\nu\}_{\nu=1}^{\infty}$ in which each of these
cubes occurs exactly once, and put
\[
 u_0=\sum_{\nu=1}^{\infty}
    \frac{2^{-\nu}}{1+\|\one_{Q_\nu}\|_X}\one_{Q_\nu},
 \qquad
 v_0=\sum_{\nu=1}^{\infty}
    \frac{2^{-\nu}}{1+\|\one_{Q_\nu}\|_{X'}}\one_{Q_\nu}.
\]
The lattice and Fatou properties imply
$u_0\in X$, $v_0\in X'$, $u_0,v_0>0$ almost everywhere, and
$\|u_0\|_X,\|v_0\|_{X'}\leq1$.

For $Z\in\{X,X'\}$, let $A_Z$ denote $A_X$ or $A_{X'}$, respectively.
For every nonnegative $h\in Z$, define
\begin{equation}\label{eq 20260812-24-noi}
 \mathcal R_Zh
 =\sum_{k=0}^{\infty}
  \frac{(\Mloc_\rho)^kh}{(2A_Z)^k}.
\end{equation}
The Fatou property and the boundedness of $\Mloc_\rho$ give
\begin{equation}\label{eq 20260812-25-noi}
 h\leq\mathcal R_Zh,
 \qquad
 \|\mathcal R_Zh\|_Z\leq2\|h\|_Z,
 \qquad
 \Mloc_\rho(\mathcal R_Zh)\leq2A_Z\mathcal R_Zh.
\end{equation}
Consequently, if $h>0$ almost everywhere, then
$\mathcal R_Zh\in A_1(\rho)$ and
$[\mathcal R_Zh]_{A_1(\rho)}\leq2A_Z$.

Fix $(F,G)\in\mathcal F$ with $G\in X$. Let $0\leq h\in X'$ satisfy
$\|h\|_{X'}\leq1$, and fix $\varepsilon>0$. Set
\[
 U=\mathcal R_X(G+\varepsilon u_0),
 \qquad
 V=\mathcal R_{X'}(h+\varepsilon v_0).
\]
Then $U,V$ are strictly positive $A_1(\rho)$ weights and
\begin{equation}\label{eq 20260812-26-noi}
 \|U\|_X\leq2(\|G\|_X+\varepsilon),
 \qquad
 \|V\|_{X'}\leq2(1+\varepsilon).
\end{equation}
Write $p=p_0$ and define $W=U^{1-p}V$. For every cube $Q$ with
$\ell(Q)\leq\rho$, the $A_1(\rho)$ inequalities yield
\begin{align*}
 \avg{W}_Q
 &\leq[U]_{A_1(\rho)}^{p-1}
    \avg{U}_Q^{1-p}\avg{V}_Q,\\
 \avg{W^{-1/(p-1)}}_Q
 &\leq[V]_{A_1(\rho)}^{1/(p-1)}
    \avg{V}_Q^{-1/(p-1)}\avg{U}_Q.
\end{align*}
Thus
\begin{equation}\label{eq 20260812-27-noi}
 [W]_{A_p(\rho)}
 \leq(2A_X)^{p-1}(2A_{X'})=:K.
\end{equation}
The same calculation shows the required local integrability of the two
functions in the $A_p(\rho)$ condition.

Let $I=\int_{\R^n}U(x)V(x)\,dx$. By
\eqref{eq 20260702-3} and \eqref{eq 20260812-26-noi},
\begin{equation}\label{eq 20260812-28-noi}
 I\leq4(\|G\|_X+\varepsilon)(1+\varepsilon).
\end{equation}
Since $G\leq U$ and $h\leq V$,
\[
 \int_{\R^n}G^pW\,dx\leq I,
 \qquad
 \int_{\R^n}h^{p'}W^{-1/(p-1)}\,dx\leq I.
\]
Weighted H\"older's inequality, \eqref{eq 20260812-22-noi}, and
\eqref{eq 20260812-27-noi} give
\[
 \int_{\R^n}Fh\,dx
 \leq\mathcal N(K)I
 \leq4\mathcal N(K)(\|G\|_X+\varepsilon)(1+\varepsilon).
\]
Taking the supremum over all such $h$, using $X''=X$ from
Lemma~\ref{lem 20260812-1-noi}, and then letting
$\varepsilon\downarrow0$ proves \eqref{eq 20260812-23-noi}.
\end{proof}

\subsection{Infinite-sequence estimates and the cases $0<q\leq1$ and $q=\infty$}
\label{section:q<1,infinity}

Corresponding to Lemma \ref{lem 20260702-3}, we have the following result.
\begin{proposition}
\label{prop 20260812-3-noi}
Let $\rho\in(0,\infty)$, let $X$ be a ball quasi-Banach function space,
and let $0<q<\infty$. Assume that \eqref{eq 20260812-3-noi} holds
with constant $C$ for every $N\in\mathbb{N}$. Then
$X\subset L^1_{\mathrm{loc}}(\mathbb{R}^n)$, and
\begin{equation}\label{eq 20260812-29-noi}
 \left\|
 \left(\sum_{j=1}^{\infty}(\Mloc_\rho f_j)^q\right)^{1/q}
 \right\|_X
 \leq C
 \left\|
 \left(\sum_{j=1}^{\infty}|f_j|^q\right)^{1/q}
 \right\|_X
\end{equation}
for every sequence $\{f_j\}_{j=1}^{\infty}$ such that
\[
\left(\sum_{j=1}^{\infty}|f_j|^q\right)^{1/q}\in X.
\]
\end{proposition}
\begin{proof}
Taking $N=1$ in \eqref{eq 20260812-3-noi}, we obtain
\[
 \|\Mloc_\rho f\|_X\leq C\|f\|_X
 \qquad(f\in X).
\]
Therefore, the argument in the proof of 
Lemma \ref{lem 20260702-3} works
to show that $X \subset L^1_{\rm loc}({\mathbb R}^n)$.
For example, instead of \eqref{eq:260813-1-Sawano}
we can show that
\begin{equation}\label{eq:260813-2-Sawano}
\Mloc_\rho f=\infty
\end{equation} 
on $Q$
if $\displaystyle\|f\|_{L^1(Q)}=\infty$ for some cube $Q$ with $\ell(Q)\leq\rho$.
Having established $X\subset L^1_{\rm loc}({\mathbb R}^n)$,
we define $F_N$, $G_N$, $F$, $G$ exactly as in the proof of Lemma \ref{lem 20260702-3}, with ${\mathcal M}$ replaced by ${\mathcal M}_{\rho}^{\rm loc}$. Then 
\[
\| G_N\|_X \le C \| F_N\|_X \le C\| F\|_X. 
\]
Since $G_N\uparrow G$, the Fatou property gives \eqref{eq 20260812-29-noi}. 
\end{proof}

Corresponding to Proposition \ref{prop 20260702-1},
we have the following negative result.
\begin{proposition}
\label{prop 20260812-4-noi}
Let $\rho\in(0,\infty)$, let $X$ be a ball quasi-Banach function space on
$\R^n$, and let $0<q\leq1$. There is no constant $C>0$ such that
\begin{equation}\label{eq 20260812-30-noi}
 \left\|
 \left(\sum_{j=1}^N(\Mloc_\rho f_j)^q\right)^{1/q}
 \right\|_X
 \leq C
 \left\|
 \left(\sum_{j=1}^N|f_j|^q\right)^{1/q}
 \right\|_X
\end{equation}
holds for every $N\in\N$ and every finite family of measurable
functions for which the quantity on the right is finite.
\end{proposition}

\begin{proof}
Since the proof of Proposition \ref{prop 20260702-1}
used compactly supported functions,
the same argument as \ref{prop 20260702-1} works.
In fact, choose $\tau>0$ such that $2\tau\le \rho$, and reexamine the proof of Proposition \ref{prop 20260702-1} with $[0,1)^n$ replaced by $[0,\tau)^n$. 
\end{proof}

The next proposition is a counterpart to Theorem \ref{thm 20260701-1}$(iv)$.
\begin{proposition}[$\ell^\infty$-valued estimates]
\label{prop 20260812-5-noi}
Let $\rho\in(0,\infty)$ and let $X$ be a ball quasi-Banach function space
continuously embedded into $L^1_{\mathrm{loc}}(\R^n)$. The following
statements are equivalent. The constant $C$ may be chosen to be the same
in \emph{$(i)$}--\emph{$(iii)$}.
\begin{enumerate}
\item[$(i)$] There is $C>0$ such that
\[
 \|\Mloc_\rho f\|_X\leq C\|f\|_X
 \qquad(f\in X).
\]
\item[$(ii)$] There is $C>0$ such that, for every $N\in\N$ and every sequence
$\{f_j\}_{j=1}^N$ satisfying
$\max\limits_{1\leq j\leq N}|f_j|\in X$,
\[
 \left\|\max\limits_{1\leq j\leq N}\Mloc_\rho f_j\right\|_X
 \leq C
 \left\|\max\limits_{1\leq j\leq N}|f_j|\right\|_X.
\]
\item[$(iii)$] There is $C>0$ such that, for every sequence
$\{f_j\}_{j=1}^{\infty}$ satisfying $\sup\limits_{j \in{\mathbb N}}|f_j|\in X$,
\[
 \left\|\sup\limits_{j \in{\mathbb N}}\Mloc_\rho f_j\right\|_X
 \leq C
 \left\|\sup\limits_{j \in{\mathbb N}}|f_j|\right\|_X.
\]
\end{enumerate}
Moreover, even when $X$ is a ball Banach function space, these equivalent conditions need not imply that $\Mloc_\rho$ is bounded on $X'$.
\end{proposition}

\begin{proof}
Go through
the same argument as Theorem \ref{thm 20260701-1}
$(iv)$ works.

For the final assertion, take $X=L^\infty(\R^n)$. Then
$\Mloc_\rho$ is bounded on $X$, whereas $X'=L^1(\R^n)$ and
$\Mloc_\rho$ is not bounded on $L^1(\R^n)$.
\end{proof}

\subsection{Dyadic grids and the sparse criterion at a fixed truncation parameter}
\label{sec 20260812-3-noi}

We use the same system $\mathcal{D}$ as Section \ref{subsection:Dyadic grids and sparse families}.
This adjacent-grid fact is standard; see \cite{Lerner2013}. Consequently,
for $\rho\in(0,\infty)$,
\begin{equation}\label{eq 20260812-32-noi}
 \Mloc_\rho f
 \leq6^n\sum_{t=1}^{3^n}{{\M_{\D^t,(6\rho)}}}f,
\end{equation}
where 
\[ {{\M_{\D,(R)}}}f(x)
 =\sup\limits_{{Q\in\D(R)}}\one_Q(x)\avg{|f|}_Q.
 \]
Here, for a dyadic grid $\D$ and $R\in(0,\infty]$, write
\begin{equation}\label{eq 20260812-4-noi}
 \D(R)=\{Q\in\D:\ell(Q)\leq R\},
 \qquad
 \D(\infty)=\D.
\end{equation}
When $\Sscr\subset\D(R)$, we also write
$\T_{\Sscr,q,(R)}:=\T_{\Sscr,q}$ and
$\A_{\Sscr,(R)}:=\A_{\Sscr}$ to make the upper bound $R$ on the side
lengths of the cubes in $\Sscr$ explicit.

\begin{lemma}[Finite dyadic sparse domination]
\label{lem 20260812-4-noi}
Let $R\in(0,\infty)$, let $\D$ be a dyadic grid, and let
$\Fscr\subset\D(R)$ be a finite collection. Set
\[
{{\M_{\Fscr,(R)} }}f(x)=\max\limits_{Q\in\Fscr}\avg{|f|}_Q\one_Q(x)
\]
when $\Fscr\neq\varnothing$, and use the convention
${{\M_{\varnothing,(R)}}}f=0$. For every locally integrable $f$, there is a
finite $1/2$-sparse collection $\Sscr\subset\Fscr$ such that
\begin{equation}\label{eq 20260812-33-noi}
 {{\M_{\Fscr,(R)}}}f\leq2{{\A_{\Sscr,(R)}}}f
 \quad\text{almost everywhere}.
\end{equation}
\end{lemma}

\begin{proof}
Apply the recursive selection in Lemma~\ref{lem 20260702-4} to
$\Fscr$. Since every selected cube belongs to $\Fscr$, one has
\[
 \Sscr\subset\Fscr\subset\D(R).
\]
The pointwise estimate follows from
\eqref{eq 20260702-17}.
\end{proof}

The proof of the next proposition follows the corresponding argument in
Proposition~\ref{prop 20260704-1} with
$\D$ replaced by $\D(6\rho)$ and the local maximal operators kept
explicit throughout.

\begin{proposition}[Sparse criterion with truncation parameter $\rho$]
\label{prop 20260812-6-noi}
Let $\rho\in(0,\infty)$ and let $X$ be a ball Banach function space. The
following statements are equivalent.
\begin{enumerate}
\item[$(i)$] The operator $\Mloc_\rho$ is bounded on both $X$ and $X'$.
\item[$(ii)$] There is a constant $C_{\mathrm{sp}}>0$ such that
\begin{equation}\label{eq 20260812-34-noi}
 \|{{\A_{\Sscr,(6\rho)}}}f\|_X\leq C_{\mathrm{sp}}\|f\|_X
\end{equation}
for every finite $1/2$-sparse subcollection $\Sscr$ of $\D(6\rho)$,
where $\D$ is an arbitrary dyadic grid, and every $f\in X$, with a
constant independent of $\D$ and $\Sscr$.
\end{enumerate}
\end{proposition}

\begin{proof}
The implication (i) $\Rightarrow$ (ii) follows from Corollary \ref{cor 20260812-1-noi} and the sparse sets $E_Q$ by duality if we argue similarly to Lemma \ref{lemm 20260810-1-noi}.

Conversely, assume (ii), namely, \eqref{eq 20260812-34-noi}.
Then the boundedness of $\Mloc$ on $X$ follows by repeating the corresponding part of Proposition \ref{prop 20260704-1} with ${\mathcal D}$ replaced by ${\mathcal D}(6\rho)$, using Lemma \ref{lem 20260812-4-noi}, and the Fatou property of $X$, and \eqref{eq 20260812-32-noi}.
Since \eqref{eq 20260812-34-noi} is symmetric with respect to $X$,
that is, we can replace $X$ with $X'$ assuming \eqref{eq 20260812-34-noi},
we see that $\Mloc_{\rho}$ is bounded also on $X'$.
\end{proof}

\subsection{From a vector-valued bound to sparse averaging}
\label{subsection:From a vector-valued bound to sparse averaging}
For a finite truncation parameter, the proof of the next proposition follows
the argument of Proposition \ref{prop 20260703-1}. It applies the vector-valued
maximal inequality to the functions
$h_Q=\avg{|f|}_Q\one_{E_Q}$, where the pairwise disjoint sets
$E_Q\subset Q$ witness sparseness. By
Corollary~\ref{cor 20260812-1-noi}, the required vector-valued estimate is
available for $\Mloc_{6\rho}$.

\begin{proposition}[A sparse $\ell^q$ estimate]
\label{prop 20260812-7-noi}
Let $\rho\in(0,\infty)$, let $X$ be a ball Banach function space, and let
$1<q<\infty$. Suppose that \eqref{eq 20260812-3-noi} holds with constant
$C$. Then there is $K>0$ such that
\begin{equation}\label{eq 20260812-35-noi}
 \|{{\T_{\Sscr,q,(6\rho)}}}f\|_X\leq K\|f\|_X
\end{equation}
for every finite $1/2$-sparse subcollection $\Sscr$ of $\D(6\rho)$,
where $\D$ is an arbitrary dyadic grid, 
{{and every $f\in X$.}}
The constant is independent of
$\D$ and $\Sscr$.
\end{proposition}

\begin{proof}
We follow the idea of Proposition \ref{prop 20260703-1}.
We use $\Mloc$ instead of $\M$.
Assume
$0<\rho<\infty$. By Corollary~\ref{cor 20260812-1-noi}, the
vector-valued estimate for $\Mloc_{6\rho}$ over finite families holds
with some constant
$C_{6\rho}$. Choose measurable sets $E_Q\subset Q$, $Q\in\Sscr$, such that
$|E_Q|\geq |Q|/2$ for every $Q\in\Sscr$ and the sets $E_Q$ are pairwise
disjoint, and put
\[
 h_Q=\avg{|f|}_Q\one_{E_Q}.
\]
Since the sets $E_Q$ are disjoint and every $Q\in\Sscr$ satisfies
$\ell(Q)\leq6\rho$,
\[
 \left(\sum_{Q\in\Sscr}|h_Q|^q\right)^{1/q}
 \leq\Mloc_{6\rho}f.
\]
Moreover, for $x\in Q$,
\[
 \Mloc_{6\rho}h_Q(x)
 \geq\avg{h_Q}_Q
 =\avg{|f|}_Q\frac{|E_Q|}{|Q|}
 \geq\frac12\avg{|f|}_Q.
\]
Therefore
\begin{align*}
 \frac12\|{{\T_{\Sscr,q,(6\rho)}}}f\|_X
 &\leq
 \left\|
 \left(\sum_{Q\in\Sscr}(\Mloc_{6\rho}h_Q)^q\right)^{1/q}
 \right\|_X\\
 &\leq C_{6\rho}
 \left\|
 \left(\sum_{Q\in\Sscr}|h_Q|^q\right)^{1/q}
 \right\|_X\\
 &\leq C_{6\rho}\|\Mloc_{6\rho}f\|_X\\
 &\leq C_{6\rho}^2\|f\|_X.
\end{align*}
This proves \eqref{eq 20260812-35-noi}.
\end{proof}

As we did in Theorem~\ref{thm 20260701-1},
if we assume that $\Mloc_\rho$ satisfies the vector-valued inequality \eqref{eq 20260812-3-noi},
we have uniform sparse averaging bound of the
operators ${{\A_{\Sscr,(6\rho)}}}$ are uniformly bounded on $X$ over all dyadic grids
and all finite $1/2$-sparse collections
$\Sscr\subset\D(6\rho)$.
\begin{corollary}[Uniform sparse averaging bound]
\label{cor 20260812-2-noi}
Let $1<q<\infty$.
Let $\rho\in(0,\infty)$.
Suppose that \eqref{eq 20260812-3-noi} holds with constant
$C$. 
The
operators ${{\A_{\Sscr,(6\rho)}}}$ are uniformly bounded on $X$ over all dyadic grids
and all finite $1/2$-sparse collections
$\Sscr\subset\D(6\rho)$.
\end{corollary}

\begin{proof}
Let $m\in\N$ satisfy
$m>q$.
We use Theorem \ref{thm 20260703-1} to have
\begin{equation}\label{eq 20260812-37-noi}
 {{\A_{\Sscr,(6\rho)}}}f
 \leq\Gamma_{q,m}(\T_{\Sscr,q})^{m}f=\Gamma_{q,m}({{\T_{\Sscr,q,(6\rho)}}})^{m}f
 \quad\text{almost everywhere}.
\end{equation}
Let $K$ be the uniform bound in
\eqref{eq 20260812-35-noi}. 
Pointwise estimate \eqref{eq 20260812-37-noi} gives
\[
 \|{{\A_{\Sscr,(6\rho)}}}f\|_X
 \leq\Gamma_{q,m}\|({{\T_{\Sscr,q,(6\rho)}}})^{m}f\|_X
 \leq\Gamma_{q,m}K^m\|f\|_X.
\]
Thus,
the
operators ${{\A_{\Sscr,(6\rho)}}}$ are uniformly bounded on $X$ over all dyadic grids
and all finite $1/2$-sparse collections
$\Sscr\subset\D(6\rho)$.
\end{proof}

\subsection{Proof of Theorem~\ref{thm 20260812-1-noi}}
\label{subsec 0813-1-noi}

We remark that the proof of Theorem~\ref{thm 20260812-2-noi}
is similar to that of Theorem~\ref{thm 20260701-2}.
We concentrate on the proof of Theorem~\ref{thm 20260812-1-noi}.
\begin{proof}[Proof of Theorem~\ref{thm 20260812-1-noi}]
Part~\emph{$(ii)$} is Proposition~\ref{prop 20260812-4-noi},
part~\emph{$(iii)$} is Proposition~\ref{prop 20260812-3-noi}, and
part~\emph{$(iv)$} is Proposition~\ref{prop 20260812-5-noi}. It remains to
prove part~\emph{$(i)$}.

Assume $(b)$. Proposition~\ref{prop 20260812-7-noi} and
Corollary~\ref{cor 20260812-2-noi} give a uniform sparse bound.
Proposition~\ref{prop 20260812-6-noi} then yields $(a)$. Thus
$(b)\Rightarrow(a)$.

Assume $(a)$, and fix $q\in(1,\infty)$. Choose any
$p_0\in(1,\infty)$. For a finite family $\{f_j\}_{j=1}^N$, put
\[
 F=\left(\sum_{j=1}^N(\Mloc_\rho f_j)^q\right)^{1/q},
 \qquad
 G=\left(\sum_{j=1}^N|f_j|^q\right)^{1/q}.
\]
Proposition~\ref{prop 20260812-1-noi} verifies the weighted hypothesis of
Proposition~\ref{prop 20260812-2-noi}, with a bound independent of $N$.
Applying extrapolation at a fixed truncation parameter to the family of all such pairs $(F,G)$
gives \eqref{eq 20260812-3-noi} on $X$. Since $q$ was arbitrary, $(c)$
holds. The implication $(c)\Rightarrow(b)$ is immediate. Hence all
three conditions are equivalent.
\end{proof}

\section*{Data availability}

Not applicable.

\section*{Author contributions}

The three authors contributed equally to the
correctness of this paper.

\section*{Acknowledgements}

The authors are thankful to Professor Emiel Lorist
for his advice on Theorem \ref{thm 20260701-1}.
This work was partly supported by MEXT Promotion of Distinctive Joint
Research Center Program JPMXP0723833165
and Osaka Metropolitan University Strategic Research Promotion Project
(Development of International Research Hubs).
 This work was supported by Grant-in-Aid for Scientific Research (C) 
 (26K14971) (Izuki), 
 Grant-in-Aid for Scientific Research (C) (23K03156, 26K14971) (Noi)
 and
 Grant-in-Aid for Scientific Research (C) (23K03156, 26K14971) (Sawano), the Japan Society for the Promotion of Science.

\end{document}